\documentclass[12pt, reqno]{amsart}
\usepackage{fullpage, url, amssymb,amsthm}
\usepackage{color}
\usepackage[colorlinks]{hyperref}
\usepackage[alphabetic, lite]{amsrefs}
\usepackage[all]{xy} 

\newtheorem{lemma}{Lemma}[section]
\newtheorem{theorem}[lemma]{Theorem}

\newtheorem{prop}[lemma]{Proposition}
\newtheorem{cor}[lemma]{Corollary}
\newtheorem{conj}[lemma]{Conjecture}

\newtheorem{claim*}{Claim}
\newtheorem{thm}[lemma]{Theorem}
\newtheorem{defn}[lemma]{Definition}

\newtheorem*{theorem_nonum}{Theorem}
\newtheorem*{cor_nonum}{Corollary}

\theoremstyle{remark}
\newtheorem*{question*}{Question}

\newtheorem{remark}[lemma]{Remark}

\newtheorem{example}[lemma]{Example}

\newcommand{\F}{{\mathbb F}}
\newcommand{\Q}{{\mathbb Q}}

\newcommand{\Z}{{\mathbb Z}}

\newcommand{\BBl}{{\mathbb B}_{\ell, \infty}}

\newcommand{\WW}{{\mathbb W}}

\newcommand{\Fbar}{{\overline{\F}}}
\newcommand{\Ebar}{{\overline{E}}}
\newcommand{\ppbar}{{\overline{\pp}}}

\newcommand{\alphabar}{{\overline{\alpha}}}
\newcommand{\betabar}{{\overline{\beta}}}

\newcommand{\Rtilde}{{\widetilde{R}}}

\newcommand{\lambdatilde}{{\widetilde{\lambda}}}
\newcommand{\deltatilde}{{\widetilde{\delta}}}

\newcommand{\pp}{{\mathfrak p}}
\newcommand{\qq}{{\mathfrak q}}

\newcommand{\calO}{{\mathcal O}}

\newcommand{\OO}{{\mathcal O}}

\newcommand{\fraka}{\mathfrak{a}}
\newcommand{\frakb}{\mathfrak{b}}
\newcommand{\frakc}{\mathfrak{c}}
\newcommand{\frakl}{\mathfrak{l}}

\newcommand{\frakp}{\mathfrak{p}}
\newcommand{\frakq}{\mathfrak{q}}
\newcommand{\frakr}{\mathfrak{r}}
\newcommand{\frakabar}{\overline{\mathfrak{a}}}

\newcommand{\frakcbar}{\overline{\mathfrak{c}}}

\newcommand{\frakA}{\mathfrak{A}}
\newcommand{\frakD}{\mathfrak{D}}

\DeclareMathOperator{\Tr}{Tr}

\DeclareMathOperator{\im}{im}

\DeclareMathOperator{\End}{End}

\DeclareMathOperator{\Lie}{Lie}
\DeclareMathOperator{\Hom}{Hom}

\DeclareMathOperator{\Gal}{Gal}

\DeclareMathOperator{\Norm}{N}

\DeclareMathOperator{\Pic}{Pic}

\DeclareMathOperator{\NN}{N}

\DeclareMathOperator{\Mat}{M}

\DeclareMathOperator{\SL}{SL}

\DeclareMathOperator{\disc}{disc}

\DeclareMathOperator{\cond}{cond}

\DeclareMathOperator{\Isom}{Isom}
\DeclareMathOperator{\unr}{unr}

\newcommand{\isom}{\cong}

\newcommand{\into}{\hookrightarrow}

\numberwithin{equation}{section}
\numberwithin{table}{section}

\newcommand{\defi}[1]{\textsf{#1}} 
\title{On singular moduli for arbitrary discriminants}

\author{Kristin Lauter}
\author{Bianca Viray}
\thanks{The second author was partially supported by Microsoft Research, a Ford dissertation year fellowship, National Science Foundation grant DMS-1002933, and the Institute for Computational and Experimental Research in Mathematics.}

\address{Microsoft Research, 1 Microsoft Way, Redmond, WA 98062, USA}
\email{klauter@microsoft.com}
\urladdr{http://research.microsoft.com/en-us/people/klauter/default.aspx}

\address{University of Washington, Department of Mathematics, Box 354350, Seattle, WA 98195, USA}
\email{bviray@math.washington.edu}
\urladdr{http://math.washington.edu/\~{}bviray}

\date{}


\begin{document}

	\begin{abstract}
		Let $d_1$ and $d_2$ be discriminants of distinct quadratic imaginary orders $\OO_{d_1}$ and $\OO_{d_2}$ and let $J(d_1,d_2)$ denote the product of differences of CM $j$-invariants with discriminants $d_1$ and $d_2$.  In 1985, Gross and Zagier gave an elegant formula for the factorization of the integer $J(d_1,d_2)$ in the case that $d_1$ and $d_2$ are relatively prime and discriminants of \emph{maximal} orders.  {To compute this formula, they first reduce the problem to counting the number of simultaneous embeddings of $\OO_{d_1}$ and $\OO_{d_2}$ into endomorphism rings of supersingular curves, and then solve this counting problem.}

        {Interestingly, this counting problem also appears when computing class polynomials for invariants of genus $2$ curves.  However, in this application, one must consider orders $\OO_{d_1}$ and $\OO_{d_2}$ that are \emph{non}-maximal.}
{Motivated by the application to genus $2$ curves,} we generalize the methods of Gross and Zagier and give a computable formula for $v_{\ell}(J(d_1,d_2))$ for any pair of discriminants $d_1\neq d_2$ and any prime $\ell>2$.
In the case that $d_1$ is squarefree and $d_2$ is the discriminant of {any} quadratic imaginary order, our formula can be stated in a simple closed form.  We also give a conjectural closed formula when the conductors of $d_1$ and $d_2$ are relatively prime.
	\end{abstract}

	\maketitle

\section{Introduction}
	Let $d_1$ and $d_2$ be discriminants of distinct quadratic imaginary orders and write
	\[
		J(d_1,d_2) := \prod_{\substack{[\tau_1], [\tau_2]\\\textup{disc }\tau_i = d_i}}
		\left(j(\tau_1) - j(\tau_2)\right),
	\]
	where $[\tau_i]$ runs over all elements of the upper half-plane with discriminant $d_i$ modulo $\SL_2(\Z)$.
In 1985, under the assumption that $d_1$ and $d_2$ are relatively prime and discriminants of \emph{maximal} orders, Gross and Zagier showed that $J(d_1, d_2)^{\frac8{w_1w_2}}$ is an integer and gave an elegant formula for its factorization.
	\begin{theorem_nonum}{\cite[Thm. 1.3]{GZ-SingularModuli}}
		Let $d_1$ and $d_2$ be two relatively prime fundamental discriminants of imaginary quadratic fields.  For any prime $p$
dividing a positive integer of the form $d_1 d_2 - x^2$, choose $i$ such that $p \nmid d_i$ and define $\epsilon(p) = \left(\frac{d_i}{p}\right)$, and extend the definition of $\epsilon$ by multiplicativity.  Then
		\[
			J(d_1, d_2)^{\frac8{w_1w_2}} =
			\pm\prod_{\substack{x\in\Z, |x| < d_1d_2\\x\equiv d_1d_2\bmod 2}} 		
				F\left(\frac{d_1d_2 - x^2}{4}\right),
		\]
		where $F(m) = \displaystyle\prod_{n|m, n>0} n^{\epsilon(m/n)}$ and $w_i$ denotes the number of roots of unity  in the quadratic imaginary order of discriminant $d_i$.
	\end{theorem_nonum}
	
	Although it is not obvious from the definition, $F(m)$ is a non-negative power of a single prime.  More precisely, $F(m)$ is a non-trivial power of $\ell$ if $\ell$ is the unique prime such that $v_{\ell}(m)$ is odd and $\epsilon(\ell) = -1$, and $F(m) = 1$ otherwise.  This fact gives a particularly nice corollary, which can be thought of as saying that $J(d_1,d_2)$ is ``highly factorizable.''
	\begin{cor_nonum}{\cite[Cor. 1.6]{GZ-SingularModuli}}
		Let $d_1$ and $d_2$ be two relatively prime fundamental discriminants.  Assume that $\ell$ divides $J(d_1, d_2)$.  Then $\left(\frac{d_1}{\ell}\right), \left(\frac{d_2}{\ell}\right) \ne 1$ and $\ell$ divides a positive integer of the form $\frac{d_1d_2 - x^2}{4}$.
	\end{cor_nonum}

    {A factorization formula for $J(d_1,d_2)$, \emph{without} assumptions on $d_1$ and $d_2$, would have important applications to the computation of minimal polynomials of Igusa invariants of CM abelian surfaces.  In contrast to the case of genus $1$ curves, these minimal polynomials do not have integer coefficients, and the denominators of the coefficients cause difficulty in the computation of these polynomials.  As explained in~\cite{LV-Igusa}, the computation of the denominators can be reduced to counting the number of pairs of elements of a given norm and trace in certain maximal orders in a quaternion algebra.  As becomes apparent in the proof in~\cite{GZ-SingularModuli}, this counting problem is essentially equivalent to a factorization formula of $J(d_1,d_2)$, where $d_1$ and $d_2$ are the discriminants of the elements in question.

    The discriminants that arise in the study of Igusa invariants are \emph{not} necessarily relatively prime or fundamental.  In particular, the results and techniques of Gross and Zagier do not suffice for the computation of the aforementioned denominators.  This motivates the focus of this paper, which is to determine to what extent a factorization formula for $J(d_1,d_2)$ holds for \emph{arbitrary} discriminants.  }



    {Before stating our main results, we must first discuss which aspects of~\cite[Thm. 1.3]{GZ-SingularModuli} we hope to preserve.} We note that the proof of~\cite[Thm. 1.3]{GZ-SingularModuli} implicitly gives a mathematical interpretation of the quantities $F(m)$.  Their interpretation is quite natural, and any nice generalization of~\cite[Thm. 1.3]{GZ-SingularModuli} to arbitrary discriminants should retain this property. (This will be explained in more detail in~\S\ref{sec:ProofOfDefnOfFm}.)

	We prove that there is a generalization of $F(m)$ such that the {expression for $J(d_1, d_2)$ in the Theorem above } holds, and that retains this {mathematical interpretation} and many other nice properties of the original definition:
	\begin{theorem}\label{thm:DefnOfFm}
		Let $d_1, d_2$ be \emph{any} two distinct discriminants.  Then there exists a function $F$ that takes non-negative integers of the form $\frac{d_1d_2 - x^2}{4}$ to (possibly fractional) prime powers.  This function satisfies
			\begin{equation}\label{eq:main_formula}
				J(d_1, d_2)^{\frac8{w_1w_2}} = \pm
				\prod_{\substack{x^2 \leq d_1d_2\\x^2\equiv d_1d_2\bmod 4}} 		
					F\left(\frac{d_1d_2 - x^2}{4}\right).
			\end{equation}
Moreover, $F(m) = 1$ unless either $(1)$ $m = 0$ and $d_2 = d_1\ell^{2k}$ for some prime $\ell$ or $(2)$ the Hilbert symbol $(d_1, -m)_{\ell} = -1$ at a unique finite prime $\ell$ and this prime divides $m$.  In both of these cases $F(m)$ is a (possibly fractional) power of $\ell$.
	\end{theorem}
	\begin{remark}
		In most cases, $F(m)$ is actually an integer prime power.  In particular, it is an integer prime power if $m$ is coprime to the conductor of $d_1$ and at least one of $d_1$ or $d_2$ is odd.  We explain this more in Theorem~\ref{thm:main}.  However, fractional powers do appear.  For instance, if $d_1 = -3$ and $d_2 = -12$, then $F(0) = 2^{1/3}.$
	\end{remark}
	\noindent Theorem~\ref{thm:DefnOfFm} easily implies a generalization of~\cite[Cor 1.6]{GZ-SingularModuli} for \emph{all} pairs of discriminants.
	\begin{cor}\label{cor:HighlyDivisible}
		Let $d_1$ and $d_2$ be \emph{any} two distinct discriminants.  Assume that $\ell$ divides $J(d_1, d_2)$.  Then either $d_2 = d_1\ell^{2k}$ for some $k$, or  $\ell$ divides a positive integer $m$ of the form $\frac{d_1d_2 - x^2}{4}$ and the Hilbert symbol $\left(d_1, -m\right)_p = \left(d_2, -m\right)_p$ is nontrivial if and only if $p = \ell$.  In particular, if $d_2 \ne d_1\ell^{2k}$ then $\left(\frac{d_1}{\ell}\right), \left(\frac{d_2}{\ell}\right) \ne 1.$
	\end{cor}
	\begin{remark}
	 	The assumption that $d_2\neq d_1\ell^{2k}$ is, in fact, necessary to conclude that $\left(\frac{d_1}{\ell}\right), \left(\frac{d_2}{\ell}\right) \ne 1$.  For example, $11$ divides $J(-7, -7\cdot11^2)$ even though $11$ is split in $\Q(\sqrt{-7})$.
	\end{remark}
	
	From the proof of Theorem~\ref{thm:DefnOfFm}, we can show that there is \emph{no} extension of the definition of the quadratic character $\epsilon$ such that $F(m) = \prod_{n} n^{\epsilon(m/n)}$ for an arbitrary pair of discriminants (see~\S\ref{sec:ProofOfDefnOfFm}, Example~\ref{ex:NoExtensionOfEpsilon}).  We will instead generalize a different expression of $F(m)$ from~\cite{GZ-SingularModuli, Dorman-SpecialValues}.  If $-d_1$ is prime, $d_2$ is fundamental, and $\textup{gcd}(d_1,d_2) = 1$, then Gross and Zagier show that $v_{\ell}(F(m))$ can be expressed as a weighted sum of the number of integral ideals in $\OO_{d_1}$ of norm $m/\ell^r$ for $r>0$.  Dorman extended this work to the case that $d_1$ is squarefree, $d_2$ is fundamental, and $\textup{gcd}(d_1,d_2) = 1$.  We derive an expression for $v_{\ell}(F(m))$ which holds for \emph{arbitrary} $d_1$ and $d_2$ and for primes $\ell > 2$ which are coprime to the conductor of $d_1$.

	\begin{theorem}\label{thm:main}
		Let $d_1$ and $d_2$ be \emph{any} two distinct discriminants, and let $m$ be a non-negative integer of the form $\frac{d_1d_2 - x^2}{4}$ and $\ell$ a fixed prime that is coprime to $\cond(d_1)$.
		
		If $m > 0$ and either $\ell > 2$ or $2$ does not ramify in both $\Q(\sqrt{d_1})$ and $\Q(\sqrt{d_2})$, then $v_{\ell}(F(m))$ can be expressed as a weighted sum of the number of certain invertible integral ideals in $\OO_{d_1}$ of norm $m/\ell^r$ for $r >0$.
		
		Moreover, if $m$ is coprime to the conductor of $d_1$, then $v_{\ell}(F(m))$ is an integer and the weights are easily computed and constant; more precisely we have
		\begin{equation}\label{eq:ValuationOfFm}
			v_{\ell}(F(m)) =
				\begin{cases}
					\frac{1}{e}\rho(m)\sum_{r\ge1}\frakA(m/\ell^r) &
						\textup{if }\ell\nmid \cond(d_2),\\
					\rho(m)\frakA(m/\ell^{1 + v(\cond(d_2))}) &
						\textup{if }\ell| \cond(d_2),\\
				\end{cases}
		\end{equation}
		where $e$ is the ramification degree of $\ell$ in $\Q(\sqrt{d_1})$ and
		\begin{align*}
			\rho(m) = &
				\begin{cases}
					0 &\textup{if } (d_1, -m)_p = -1 \textup{ for }p|d_1, p \nmid f_1\ell,\\
					2^{\#\{p|(m, d_1) : p\nmid f_2
					\textup{ or } p= \ell\}} & \textup{otherwise},
				\end{cases}\\
			\frakA(N) = & \#\left\{
				\begin{array}{ll}
					& \Norm(\frakb) = N, \frakb \textup{ invertible},\\
					\frakb\subseteq\OO_{d_1} : & p\nmid\frakb
					\textup{ for all }p | (N, f_2), p\nmid \ell d_1\\
					& \frakp^3\nmid\frakb\textup{ for all }
					\frakp|p|(N, f_2, d_1), p\ne\ell
				\end{array}
				\right\}.
		\end{align*}

	If $m = 0$, then either $v_{\ell}(F(0)) = 0$ or $d_2 = d_1\ell^{2k}$ and
		\[
			v_{\ell}(F(0)) = \frac{2}{w_1}\cdot\#\Pic(\OO_{d_1}).
		\]
	\end{theorem}
	Under the same assumptions as in Theorem~\ref{thm:main}, formula~\eqref{eq:ValuationOfFm} has an equivalent formulation as a product of local factors, see Proposition~\ref{prop:LocalFactors} in~\S\ref{sec:CountingEndomorphisms}.
	
	Theorems~\ref{thm:DefnOfFm} and~\ref{thm:main} combine to give a formula for $v_{\ell}(J(d_1,d_2))$ for any pair of discriminants $d_1\neq d_2$ and any prime $\ell>2$.  Under certain conditions this formula simplifies further:

	\begin{cor}
		Let $\ell$ be a prime and let $d_1, d_2$ be \emph{any} two distinct discriminants.
		Assume that either $d_1$ is squarefree or $\ell>2$ and for all $x\equiv d_1d_2\bmod 2$ with $x^2<d_1d_2$, we have either
		$$\gcd\left(\cond(d_1),\frac{d_1d_2 - x^2}{4}\right) = 1,
		\; \; \text{or} \; \;
			\left(d_1, \frac{x^2 - d_1d_2}{4}\right)_p = -1 \textup{ for some }p\ne \ell.$$
		Then we have
		\begin{equation}\label{eq:valJd1d2}
			v_{\ell}(J(d_1,d_2)^\frac{8}{w_1w_2}) = H +
			\sum_{\substack{x^2 < d_1d_2\\x^2\equiv d_1d_2\bmod{4}}}
			\begin{cases}
				\frac{1}{e}\rho(m_x)\sum_{r\ge1}\frakA(m_x/\ell^r) &
					\textup{if }\ell\nmid \cond(d_2),\\
				\rho(m_x)\frakA(m_x/\ell^{1 + v(\cond(d_2))}) &
					\textup{if }\ell| \cond(d_2),\\
			\end{cases}
		\end{equation}
		where $m_x := \frac{d_1d_2 - x^2}{4}$ and $H = 0$ unless $d_2 = d_1\ell^{2k}$ for some $k >0 $, in which case $H = \frac{2}{w_1}\cdot\#\Pic(\OO_{d_1})$.
		%
	\end{cor}
	We conjecture that Theorem~\ref{thm:main} holds even in the case that $\ell = 2$ ramifies in both $\Q(\sqrt{d_1})$ and $\Q(\sqrt{d_2})$.  However, the existence of multiple quadratic ramified extensions of $\Q_2^{\unr}$ causes difficulty in one of the steps of the proof, namely the proof of Proposition~\ref{prop:LieToEndo}.  It may be possible to get around this difficulty in our approach by a long and detailed case-by-case analysis.  We did not undertake this analysis, and it would be interesting to determine a better method.

	More generally, the local factor description of~\eqref{eq:ValuationOfFm} that is given in~\S\ref{sec:CountingEndomorphisms} suggests a conjecture for any pairs of discriminants whose conductors are relatively prime.
	\begin{conj}
		Let $d_1$ and $d_2$ be two distinct discriminants with relatively prime conductors.  Write $f$ for the product of the two conductors and for any prime $p$, let $d_{(p)}\in\{d_1,d_2\}$ be such that $p\nmid\cond(d_{(p)})$.  Then, for any prime $\ell$
		\[
			v_{\ell}(J(d_1,d_2)^2) = H +
			\sum_{\substack{x^2 < d_1d_2\\x^2\equiv d_1d_2\bmod{4}}}
			\epsilon_{\ell}(x)\prod_{p|m_x, p\ne\ell}
			\begin{cases}
				1 + v_p(m) & \left(\frac{d_{(p)}}{p}\right) = 1, p\nmid f,\\
				2 & \left(\frac{d_{(p)}}{p}\right) = 1, p| f, \textup{ or}\\
				& p|d_{(p)}, (d_{(p)}, -m)_p = 1, p\nmid f\\
				1 & \left(\frac{d_{(p)}}{p}\right) = -1, p\nmid f, v_p(m)\textup{ even} \textup{ or}\\
				& p|d_{(p)}, (d_{(p)}, -m)_p = 1, p |f , v_p(m) = 2\\
				0 & \textup{otherwise},
			\end{cases}
		\]
		where $H$, $m_x$ are as above and
		\[
			\epsilon_{\ell}(x) =
			\begin{cases}
				v_{\ell}(m_x) & \textup{if }\ell\nmid f, \ell|d_{\ell}\\
				\frac12(v_{\ell}(m_x) + 1), &
					\textup{if }\ell\nmid fd_{(p)}, v_{\ell}(m) \textup{ odd},\\
				0 & \textup{if }\ell\nmid d_{(p)}, v_{\ell}(m) \textup{ even},\\
				1 & \textup{otherwise}.
			\end{cases}
		\]
	\end{conj}	
	\noindent We verified this conjecture with \texttt{Magma} for all pairs of discriminants with relatively prime conductors and $|d_i| < 250$.

	\subsection{Related previous work}
		In 1989, Kaneko\cite{Kaneko} generalized part of~\cite[Cor. 1.6]{GZ-SingularModuli} to arbitrary discriminants.  More precisely, he proved that if $d_1$ and $d_2$ are arbitrary discriminants and $\ell$ is a prime dividing $J(d_1,d_2)$, then $\ell$ divides a positive integer of the form $\frac14(d_1d_2 - x^2).$  However, Kaneko did not obtain the stronger statement given in Corollary~\ref{cor:HighlyDivisible}.
		
		To the best of our knowledge, the only previous generalization of Theorem~\cite[Thm. 1.3]{GZ-SingularModuli}, conjectural or otherwise, was given in 1998 by Hutchinson.  Hutchinson put forth conjectural extensions of the Gross-Zagier formula~\cite{Hutchinson} in the case that the gcd of $d_1$ and $d_2$ is supported at a single prime $p$ that does not divide either conductor.  While his formulations are very different from ours, we checked that the two formulas agree in many cases.
		
		One of the contributions of the present paper is a generalization of Dorman's theory of maximal orders in a quaternion algebra with an optimal embedding of a maximal imaginary quadratic order~\cite{Dorman-Orders}.  We generalize this theory to include imaginary quadratic orders which are \emph{not} maximal.  The first author, together with Goren, generalized Dorman's work in a different direction, to higher-dimensional abelian varieties, by giving a description of certain orders in a quaternion algebra over a totally real field, with an optimal embedding of the maximal order of a CM number field~\cite{GL}.  Since that work does not apply to optimal embeddings of {non-maximal} orders, it is neither weaker nor stronger than the generalization we give in this paper.

	\subsection{Applications}
		The results and techniques from Gross and Zagier's paper have had a number of applications over the years.  For instance,~\cite[Cor. 1.6]{GZ-SingularModuli} gives simple conditions which ensure that certain values of the $j$ function are relatively prime to all sufficiently large primes.  This has been used to bound the number of rational points on certain modular curves~\cite[Thm. 6.2]{Parent} and to determine which twist of an elliptic curve to use for the CM method~\cite[pp.554-555]{RubinSilverberg}. In a different direction, the techniques of the paper have applications to CM liftings of supersingular elliptic curves.  More precisely, work of Elkies~\cite{Elkies} combined with~\cite{GZ-SingularModuli} shows that liftings of supersingular elliptic curves over $\Fbar_\ell$ to elliptic curves with CM by a maximal order $\OO$ are in bijection with embeddings of $\OO$ into maximal orders of $\BBl$, the quaternion algebra ramified at $\ell$ and $\infty$.  Furthermore, the work of Gross and Zagier together with work of Dorman~\cite{Dorman-Orders} gives a count for the number of such embeddings.
		
		We expect the results in this paper to lead to similar applications.  Indeed, Corollary~\ref{cor:HighlyDivisible} has already been used in work of Dose, Green, Griffin, Mao, Rolen, and Willis to study integrality properties of values at CM points of a certain non-holomorphic modular function~\cite{AWS}.

        In a different direction, as mentioned above, new applications in arithmetic intersection theory are given by the present authors in~\cite{LV-Igusa}, where the results on counting embeddings of non-maximal orders presented here in Sections \ref{sec:QuaternionOrders} and~\ref{sec:CountingEndomorphisms} are used to give formulas for the denominators of Igusa class polynomials.  We note that in both of these examples the results and techniques of Gross and Zagier and Dorman are not sufficient; results on non-fundamental discriminants and pairs of discriminants with a common factor are needed.

	\subsection{Outline}
		We prove Theorem~\ref{thm:DefnOfFm} in \S\ref{sec:ProofOfDefnOfFm}.  The rest of the paper will focus on the proof of Theorem~\ref{thm:main}.  In~\S\ref{subsec:HighLevelStrategy}, we give a high-level overview of the whole proof and explain the differences between the general case and the cases treated in~\cite{GZ-SingularModuli, Dorman-SpecialValues}.  In~\S\ref{subsec:DetailedOutline} we explain how various propositions and theorems come together to prove Theorem~\ref{thm:main}, and point the reader to the individual sections where each proposition or theorem is proved.

	\subsection{Notation}
		Throughout, $\ell$ will denote a fixed prime. By \defi{discriminant} we mean a discriminant of a quadratic imaginary order.  We say a discriminant is \defi{fundamental at a prime $p$} if the associated quadratic imaginary order is maximal at $p$, and we say a discriminant is \defi{fundamental} if it is fundamental at all primes $p$.
		
		For a discriminant $d$, we write $f$ for the conductor of $d$, and let $\widetilde{d}$ denote $d\ell^{-2v_{\ell}(f)}$.  Note that $\widetilde{d}$ is fundamental at $\ell$.  For most of the paper, we will concern ourselves with two fixed distinct discriminants $d_1$, $d_2$; in this case, the above quantities will be denoted $f_1, f_2$ and $\widetilde{d_1}, \widetilde{d_2}$ respectively.  We write $s_i := v_{\ell}(f_i).$

		We denote quadratic imaginary orders by $\OO$ and write $\OO_d$ for the quadratic imaginary order of discriminant $d$.  We set $w_d := \#\OO_d^\times$ and $\widetilde{w}_d := \#\OO_{\widetilde{d}}^\times$. The notation $w_i$, $\widetilde{w}_i$ will refer to the special case $d = d_i$.  If $\ell$ is ramified in $\OO$, then $\frakl$ will denote the unique prime ideal in $\OO$ lying above $\ell$.  We write $\frakD$ for the principal ideal generated by $\sqrt{d}$.  In~\S\ref{sec:Background}, we will give background on quadratic imaginary orders, and fix some more notation there.
		
		Let $H_d$ denote the ring class field of $\OO_d$, and let $\widetilde{H}_{d_i} := H_{\widetilde{d_i}}$.  Let $\mu_i$ denote a prime of $\OO_{H_{d_i}}$ lying over $\ell$.  When $\ell\nmid f_1$, we let $\WW$ denote $\OO_{H_{d_1},\mu_1}$ and let $\pi$ denote a uniformizer of $\WW$.

\section*{Acknowledgements}
	We thank Drew Sutherland for pointing out the work of Hutchinson.
The second author would like to thank Benedict Gross, Benjamin Howard, Bjorn Poonen, Joseph Rabinoff, and Michael Rosen for helpful conversations and Jonathan Lubin for a discussion regarding part of the proof of  Proposition~\ref{prop:IsomToLie}.  We are also grateful to {Brian Conrad}, Bjorn Poonen, and {the anonymous referees} for comments improving the exposition.

\section{Proof of Theorem~\ref{thm:DefnOfFm}}\label{sec:ProofOfDefnOfFm}

	By~\cite[\S\S5,6]{SerreTate}, there exists a number field $K$ such that, for every prime $q$ and every $[\tau]$ of fixed discriminant $d$, there exists an elliptic curve $E/\OO_K$ with good reduction at all $\qq|q$ such that $j(E) = j(\tau)$.  More specifically, we may take $K$ to be the ring class field of $\OO_d$, unless $d = -3p^{2k}$ or $-4p^{2k}$ for some prime $p$ and positive integer $k$.  In that case, we may assume that $K$ is a finite extension of the ring class field of $\OO_d$ ramified at $p$.
	
	For each $i$, let $L_i$ be the ring class field of $\OO_{d_i}$, if $d_i \ne-3p^{2k}, -4p^{2k}$ for some prime $p$.  If $d_i = -3p^{2k}$ or $-4p^{2k}$, let $L_i$ be the minimal (finite) Galois extension of the ring class field of $\OO_{d_i}$ such that the above properties hold.  Let $L$ be the compositum of $L_1$ and $L_2$ and let $\OO_L$ be the ring of integers.  Then, by the discussion above, for every $[\tau_i]$ of discriminant $d_i$ and every prime $q$, there exists an elliptic curve $E/\OO_L$ with good reduction at all $\qq|q$ such that $j(E) = j(\tau_i)$; we call let $E(\tau_i)$ denote such an elliptic curve.
	
	Fix a rational prime $\ell$ and a prime $\mu$ of $\OO_L$ lying over $\ell$.  Let $A$ be the ring of integers of $L_{\mu}^{\textup{unr}}$.  Then, by~\cite[Prop 2.3]{GZ-SingularModuli}, we have
\[
	v_{\mu}(j(\tau_1) - j(\tau_2)) =
	\frac12\sum_{n}\#\Isom_{A/\mu^n}(E(\tau_1), E(\tau_2))
\]
for all $[\tau_i]$ of discriminant $d_i$.  Write $E_i$ for $E(\tau_i)$.  Since $\tau_i$ is an algebraic number of discriminant $d_i$, we have an isomorphism $\OO_{d_i}\isom\End(E_i)$, and an embedding
\[
	\OO_{d_i}\isom\End(E_i)\hookrightarrow\End_{A/\mu}(E_i).
\]
For any $g\in \Isom_{A/\mu}(E_1, E_2)$ we obtain an isomorphism $\End_{A/\mu}(E_2)\isom\End_{A/\mu}(E_1)$ by conjugating by $g$.  Thus we have an embedding (that depends on the choice of $g$) of $\OO_{d_2}\hookrightarrow\End_{A/\mu}(E_1)$.

\begin{defn}
	Let $\iota\colon\OO\hookrightarrow R$ be a map of $\Z$-modules.  We say this map is \defi{optimal at $p$} if
	\[
		\left(\iota(\OO)\otimes\Q_p\right)\cap R = \iota(\OO),
	\]
	where the intersection takes place inside of $R\otimes\Q_p$
\end{defn}
\begin{prop}\label{prop:optimal_embedding}
	Let $E$ be an elliptic curve over $A$ that has good reduction and that has CM by an order $\OO$.  Write $\Ebar_0$ for the reduction of $E$.  Then the embedding
	\[
		\End(E) \hookrightarrow \End(\Ebar_0)
	\]
	is optimal at all primes $p\ne\ell$.  It is optimal at $\ell$ if and only if $\OO$ is maximal at $\ell$.
\end{prop}
\begin{proof}
	If $E$ has ordinary reduction, then this is a well-known result, see, for example~\cite[\S13, Thm. 12]{Lang-EllipticFunctions}.  Assume that $E$ has supersingular reduction.  We write $\Ebar_n$ for $E\bmod\mu^{n+1}$ and let $\phi \in\End(\Ebar_0)$ be the image of a generator of $\OO$.  By~\cite[Chap. II, Lemma 1.5]{Vigneras} there is a unique  maximal order in $\End(\Ebar_0)\otimes\Q_{\ell}$ which consists of all integral elements.  Therefore, the order $(\End(E)\otimes\Q)\cap\End(\Ebar_0)$ must be maximal at $\ell$, and so the embedding $\End(E)\hookrightarrow \End(\Ebar_0)$ is optimal at $\ell$ if and only if $\End(E) = \OO$ is maximal at $\ell$.
	
	By the Grothendieck existence theorem~\cite[Thm 3.4]{Conrad-GZ}, we have that
	\[
		\End(E) \to \varprojlim \End(\Ebar_n)
	\]
	is a bijection.  So to complete the proof, it suffices to show that any element of $\End(E)\cap \Q(\phi)$ lifts to an element of $ \End(\Ebar_n)\otimes\Z[1/\ell]$ for all $n\geq 0$.  Let $\psi\in\End(\Ebar_0)\cap\Q(\phi)$.  We may write $\psi := \frac{1}{f\ell^k}\left(a + b\phi\right)$ for some $a,b,f\in\Z$, $k\in \Z_{\geq0}$, where $\ell\nmid f$.  Since $\phi \in \im(\End(E)\to\End(\Ebar_0))$, both $\phi$ and $a + b\phi$ are endomorphisms of $\Ebar_n$, and of $\Gamma_n$, the $\ell$-divisible group of $E_n$.  The endomorphism ring $\End(\Gamma_n)$ is a $\Z_{\ell}$-module, so $\ell^k\psi = \frac1f\left(a + b\phi\right)$ is in $\End(\Gamma_n)$.  Thus, by Serre-Tate~\cite[Thm 3.3]{Conrad-GZ}, $\psi \in \End(\Ebar_n)\otimes\Z[1/\ell].$  This completes the proof.
\end{proof}

	This proposition shows that $\ell|J(d_1,d_2)$ only if there exists a $\tau_1$ such that there is an optimal embedding $\OO_{\widetilde{d}_2}\hookrightarrow\End_{A/\mu}(E(\tau_1))$.  Given such an optimal embedding, we can consider the sub-order $R$ of $\End_{A/\mu}(E_1)$ that is generated by the images of $\OO_{d_1}$ and $\OO_{d_2}$.  A calculation shows that the discriminant of $R$ is $\left(\frac{d_1d_2 - x^2}{4}\right)^2$ for some $x\in\Z$ with $x^2\leq d_1d_2$ and $d_1d_2\equiv x^2\bmod 4$.

	For any non-negative integer $m$ of the form $\frac{d_1d_2 - x^2}{4}$, we define $F(m)$ to be the unique ideal in $\OO_L$ such that for all rational primes $\ell$ and all primes $\mu|\ell$ in $\OO_L$
\[
	v_{\mu}(F(m)) = \frac4{Cw_1w_2}
	\sum_{\substack{[\tau_i]\\\textup{disc}(\tau_i) = d_i}}
	\sum_{n\geq 1} \#\left\{f\in\Isom_{A/\mu^n}(E(\tau_1), E(\tau_2)):
	\textup{disc}(R) = m^2\right\},
\]
where $R$ is the suborder of $\End_{A/\mu}(E(\tau_1))$ described above and $C = 1$ if $4m = d_1d_2$ and $C = 2$ otherwise.  (The presence of this $C$ scalar is to agree with the convention set in~\cite{GZ-SingularModuli}.  Since Gross and Zagier take the product over $x$-values, for every $m \neq d_1d_2/4$ the value $F(m)$ appears twice in the product, once from $x$ and once from $-x$.)  From this definition and the previous discussion it is clear that
\[
	J(d_1,d_2)^2 = \pm
	\prod_{\substack{x^2\leq d_1d_2\\x^2\equiv d_1d_2\bmod 4}}
	F\left(\frac{d_1d_2 - x^2}{4}\right).
\]
In addition, the condition defining $F(m)$ is Galois invariant, so $v_{\mu}(F(m)) = v_{\mu'}(F(m))$ for any $\mu,\mu'$ lying over the same prime $\ell$.

	Assume that $F(m)$ is non-trivial.  So there exists a prime $\ell$ and an elliptic curve $\overline{E}/\Fbar_{\ell}$ with embeddings $\iota_i\colon\OO_{d_i}\hookrightarrow \End(E)$ such that the images of $\iota_1$ and $\iota_2$ generate an order $R$ of discriminant $m^2$.  The endomorphism ring $\End(\overline{E})$ is either an order in a quadratic imaginary field, or a maximal order in the quaternion algebra $\BBl$ ramified only at $\ell$ and $\infty$~\cite[Chap. 13 \S2]{Lang-EllipticFunctions}.  Since a maximal order in a quaternion algebra has no suborders of rank $3$, if $m = 0$, then $R$ must be a rank $2$ $\Z$-module.  Since $\OO_{\widetilde{d}_i}$ is optimally embedded in $\End_{A/\mu}(E(\tau_1))$ this implies that $\widetilde{d}_1 = \widetilde{d}_2$, i.e. that $d_2 = d_1\ell^{2k}$ for some $k\in\Z_{\neq 0}$.

	Now consider the case when $m$ is nonzero.  Then $R$ is rank $4$, and thus $R\otimes_{\Z}{\Q}$ is $\BBl$, the quaternion algebra ramified only at $\ell$ and $\infty.$  A straightforward calculation, which does not depend on $\ell$, shows that
\[
	R\otimes_{\Z}{\Q} \isom
	\frac{\Q\langle i, j\rangle}{i^2 = d_1, j^2 = -m, ij = -ji }
	\isom\frac{\Q\langle i, j\rangle}{i^2 = d_2, j^2 = -m, ij = -ji }.
\]
Since this quaternion algebra is ramified only at $\ell$ and $\infty$, this implies that the Hilbert symbol $(d_1,-m)_p = (d_2, -m)_p = 1$ if and only if $p\neq \ell$.
 In addition, $\ell$ divides the reduced discriminant of any order in $\BBl$, so $\ell|m$.  Since $F(m)$ is a Galois invariant fixed ideal supported only at primes lying over $\ell$, we may consider $F(m)$ to be just a fractional power of $\ell$.  This completes the proof of Theorem~\ref{thm:DefnOfFm}.  \qed

\begin{example}\label{ex:NoExtensionOfEpsilon}
	Assume that there is a quadratic character $\epsilon$ that is defined on every prime $p$ that divides a positive integer of the form $\frac14(d_1d_2 - x^2)$.  In addition, we assume that if $p\nmid\textup{gcd}(d_1,d_2)$ and if $p\nmid f_1f_2$, then $\epsilon(p)$ agrees with the definition in~\cite[Thm. 1.3]{GZ-SingularModuli}, and that
	\begin{equation}\label{eq:ProductFormulaForFm}
		F\left(m\right) =
		\prod_{n|m, n>0} n^{\epsilon(m/n)}, \quad\textup{where } m = \frac{d_1d_2 - x^2}{4}.
	\end{equation}

	Consider the following example. Let $d_1 = -3\cdot7\cdot11$ and $d_2 = 5d_1$.  We will study the cases where $x^2 = 33^2$ and $99^2$, i.e. when $m = 3^2 11^2 61,$ and $3^2 11^2 59$ respectively.  Equation~\eqref{eq:ProductFormulaForFm} shows that
	\begin{align*}
		v_3(F(3^2 11^2 61)) = v_3(F(3^2 11^2 59)) &	
		= (4 + 2\epsilon(3))(1 + \epsilon(11) + \epsilon(11)^2),\\
		v_{11}(F(3^2 11^2 61)) = v_{11}(F(3^2 11^2 59)) &	
		= (4 + 2\epsilon(11))(1 + \epsilon(3) + \epsilon(3)^2),\\
		v_{61}(F(3^2 11^2 61)) = v_{59}(F(3^2 11^2 59)) &	
		= (1 + \epsilon(3) + \epsilon(3)^2)(1 + \epsilon(11) + \epsilon(11)^2).
	\end{align*}
	
	On the other hand, we can also calculate $F(m)$ using the results in the proof of Theorem~\ref{thm:DefnOfFm}.  Theorem~\ref{thm:DefnOfFm} implies that $F(3^2 11^2 61)$ is supported only at $3$ and that $F(3^211^259)$ is supported only at $11$.  Moreover, the proof gives a method of determining whether $F(m)$ is a \emph{nontrivial} power of $3$, or $11$ respectively, and the calculation shows that this is indeed the case.  But the above expressions show that either $3$ divides both $F(3^2 11^2 61)$ and $F(3^211^259)$, or it divides neither, which gives a contradiction.
\end{example}

\section{Overview of proof of  Theorem~\ref{thm:main}}\label{sec:Outline}

	Henceforth, we fix a non-negative integer $m$ of the form $\frac{d_1d_2 - x^2}{4}$ and a prime $\ell$, and we assume that $\ell\nmid f_1$.
We retain the notation fixed in~\S\ref{sec:ProofOfDefnOfFm}.  Recall from~\S\ref{sec:ProofOfDefnOfFm}, that
\[
	v_{\ell}(F(m)) = e(\mu/\ell)^{-1}\frac4{Cw_1w_2}
	\sum_{\substack{[\tau_i]\\\textup{disc}(\tau_i) = d_i}}
	\sum_{n\geq 1} \#\left\{f\in\Isom_{A/\mu^n}(E(\tau_1), E(\tau_2)):
	\textup{disc}(R) = m^2\right\}.
\]

	\subsection{High-level strategy}\label{subsec:HighLevelStrategy}
	As discussed in~\S\ref{sec:ProofOfDefnOfFm}, an element of $f\in\Isom_{A/\mu^n}$ gives rise to an embedding of $\OO_{d_2}\hookrightarrow\End_{A/\mu^n}(E(\tau_1))$ that is optimal away from $\ell$.  We first show that the problem of counting elements in $\cup_{\tau_2}\Isom_{A/\mu^n}(E(\tau_1), E(\tau_2))$ is equivalent to counting elements in $\End_{A/\mu^n}(E(\tau_1))$ that have a fixed degree and trace and have a fixed action on the Lie algebra.  To compute these elements we give detailed constructions of the endomorphism rings $\End_{A/\mu^n}(E(\tau_1))$, and show that the endomorphisms of a fixed degree and trace are in a finite-to-$1$ correspondence with ideals in $\OO_{d_1}$ of a certain norm.  We then classify how many of these endomorphisms have the desired action on the Lie algebra.
	
	This high-level strategy is the same as that employed by Gross and Zagier in the case that $-d_1$ is prime, $d_2$ is fundamental, and $\textup{gcd}(d_1,d_2) = 1$, and, soon after, by Dorman in the case that $d_1$ is squarefree, $d_2$ is fundamental, and $\textup{gcd}(d_1,d_2) = 1$.  However, the general case presents significantly more technical difficulties, which is perhaps not surprising, as it has been over 20 years since Gross-Zagier and Dorman published their papers.
	
	First, the straightforward generalizations of the constructions of the endomorphism rings $\End_{A/\mu^n}(E(\tau_1))$ given in~\cite{GZ-SingularModuli, Dorman-Orders} to the case where $d_1\equiv0\pmod 4$ and $\mu$ a prime of characteristic $2$ completely fail, \emph{even if} $d_1$ is fundamental.  If $d_1$ is not fundamental, then many arguments in~\cite{GZ-SingularModuli, Dorman-Orders} fail since the localizations of $\OO_{d_1}$ are not necessarily discrete valuation rings.  In addition, the descriptions of the constructions given in~\cite{Dorman-Orders} in the case that $d_1$ is squarefree and $\ell$ is ramified were incomplete and the proofs were omitted.  We give a construction that works generally for all $d_1$, not necessarily fundamental, and all primes $\ell$ regardless of the splitting behavior of $\ell$; see \S\ref{sec:QuaternionOrders} for more details.
	
	The next difficulty arose in studying the elements of the endomorphism rings; this study takes place in~\S\ref{sec:Main}.  The elements of these endomorphisms rings give rise to a study of ideals in $\OO_{d_1}$, an order that is not necessarily maximal.  This leads to difficulties in two ways.  The first is that, in the Picard group of a \emph{non-maximal} order, we can no longer assume that every $2$-torsion element is represented by a ramified prime ideal.  The second is that there are many more invertible ideals of order $p^n$ for $n$ sufficiently large when $p|f_1$ than in the usual case.
	
	The last point of difficulty is in determining the action on the Lie algebra in the case that $d_1$ and $d_2$ share a common factor.  If $d_1$ and $d_2$ are relatively prime, the action is almost trivial to compute.  Indeed, this step in~\cite{GZ-SingularModuli} was dealt with by a one line argument.  The general case is significantly more involved; see \S\ref{sec:LieToEndo} for more details.
	
	\subsection{Detailed outline of proof}\label{subsec:DetailedOutline}
	    Assume that $d_1$ is fundamental at $\ell$ and that either $\ell\nmid\gcd(d_1,\widetilde{d_2})$ or that $\ell>2$.  Let $\WW$ be the maximal order of $\Q_{\ell}(\sqrt{d_1})^{\unr}$; we write $\pi$ for the uniformizer of $\WW$.  Let $E/\WW$ be an elliptic curve with CM by $\OO_{d_1}$ and with good reduction.  {Note that the formula for $v_{\ell}(F(m))$ is local, and remains unchanged by unramified extensions so we may replace $L$ with $L_{\mu}^{\textup{unr}}$.  Further we may assume that $L$ is the minimal extension of $\Q_{\ell}^{\textup{unr}}$ such that every elliptic curve with CM by $\OO_{d_i}$ for $i = 1,2$ has good reduction at $\mu$.}
	
	As discussed above, we will relate the elements of $\Isom_{A/\mu^n}$ to certain endomorphisms.  {Fix $\deltatilde\in A$ a fixed root of $4z^2 - 4\widetilde{d}_2z + {\widetilde{d_2}^2 - \widetilde{d_2}}$, and $\delta := \ell^{s_2}\deltatilde - \frac12\ell^{s_2}\widetilde{d}_2(1 - \ell^{s_2})$. (Note that with this definition $\delta$ satisfies $4z^2 - 4{d}_2z + {{d_2}^2 - {d_2}}$.)} Now consider the following subsets of $\End_{A/\mu^n}(E)$.
	\begin{align*}
		S_{n}(E/A) & :=
		\left\{
		\begin{array}{ll}
			\phi\in\End_{A/\mu^n}(E) :
			& \phi^2 - d_2\phi +  \frac14(d_2^2 - d_2) = 0,\\
			&\Z[\phi]\hookrightarrow\End_{A/\mu}(E)
			\textup{ optimal away from }\ell
		\end{array}\right\},\\
		S_n^{\Lie}(E/A) & :=
			\left\{
			\begin{array}{rl}
				\phi\in S_{n}(E/A): &
				\widetilde{\phi} := \ell^{-s_2}\phi_0 + \frac12\widetilde{d}_2(1 - \ell^{s_2})
				= \deltatilde \textup{ in }\Lie(E\bmod\mu)\\
				& \phi = \delta \textup{ in }\Lie(E\bmod\mu^n)
			\end{array}\right\}.
	\end{align*}
	
In~\S\ref{sec:IsomToLie} we show that
	\begin{equation}\label{eqn:IsomToLie}
		\sum_{\tau_2}\#\Isom_{A/\mu^n}(E, E(\tau_2)) =
		\begin{cases}
			0 & \textup{if }\ell|f_2\textup{ and }n>1,\\
			\widetilde{w_2}\frac{\#\Pic(\OO_{d_2})}{\#\Pic(\OO_{\widetilde{d_2}})} \#S_{n}^{\Lie}(E) & \textup{otherwise}.
		\end{cases}
	\end{equation}
	
	In order to relate $v_{\ell}(F(m))$ to the cardinalities of $S_n^{\Lie}$, we must first partition these sets by $m$; recall that $v_{\ell}(F(m))$ counts isomorphisms that give rise to a suborder of $\End(E)$ of discriminant $m^2$.  We define:
	\begin{align*}
		S_{n, m}(E/A) & := \left\{\phi\in S_{n}(E/A):
		\textup{disc}(\OO_{d_1}[\phi]) = m^2\right\},\\
		S_{n,m}^{\Lie}(E/A) & := S_{n,m}(E/A)\cap S_n^{\Lie}(E/A).
	\end{align*}
{	From the definition of $F(m)$ and~\eqref{eqn:IsomToLie}, we see that
	\begin{equation}\label{eqn:ValToLie}
		\frac{Cw_1}{4}e(\mu/\ell)v_{\ell}(F(m)) =
		\begin{cases}
			\sum_{\tau_1}\sum_{n\geq1}\#S_{n,m}^{\Lie}(E(\tau_1)/A) &
			\textup{if }\ell\nmid f_2,\\
			\frac{\widetilde{w}_2}{w_2}\frac{\#\Pic(\OO_{d_2})}{\#\Pic(\OO_{\widetilde{d_2}})}
			\sum_{\tau_1}\#S_{1,m}^{\Lie}(E(\tau_1)/A) & \textup{if }\ell| f_2
		\end{cases}
	\end{equation}
	where $C = 1$ if $4m = d_1d_2$ and $C=2$ otherwise.

    Next we relate $\#S_{n,m}^{\Lie}(E/A)$ to $\#S_{n}(E/A)$ using the following result from~\S\ref{sec:LieToEndo}:
		\begin{align}
			\text{if }\ell|f_2, \quad& \#S_{1}^{\Lie}(E(\tau_1)/A) =
				\begin{cases}
					\#S_{1}(E(\tau_1)/\WW) & \textup{if }\ell|\widetilde{d}_2,\\
					\frac{1}2\#S_{1}(E(\tau_1)/\WW) &
						\textup{if }\ell\nmid \widetilde{d}_2,
				\end{cases}\label{eqn:Lie1ToEndo}\\
			\text{and if }\ell\nmid f_2, \quad&\sum_n \#S_{n}^{\Lie}(E(\tau_1)/A) =
				\begin{cases}
					\sum_n\#S_n(E(\tau_1)/\WW) & \textup{if }\ell\nmid d_1, \ell| \widetilde{d}_2,\\
                    \frac12 \sum_n\#S_{n}(E(\tau_1)/\WW) & \textup{if }\ell|d_1\textup{ or }\ell\nmid\widetilde{d}_2.
				\end{cases}\label{eqn:LieToEndo}
		\end{align}
        Together with~\eqref{eqn:ValToLie}, these expressions yield
        \begin{equation}\label{eqn:ValToEndo}
    		\frac{Cw_1}{4}e(\mu/\ell)v_{\ell}(F(m)) =
    		\begin{cases}
    			\sum_{\tau_1}\sum_{n\geq1}\#S_{n,m}(E(\tau_1)/\WW)
                    & \textup{if }\ell\nmid f_2d_1
                        \textup{ and }\ell|\widetilde{d}_2\\
    			\frac12\sum_{\tau_1}\sum_{n\geq1}\#S_{n,m}(E(\tau_1)/\WW)
                    & \textup{if }\ell\nmid f_2 \textup{ and either }
                        \ell\nmid\widetilde{d}_2\textup{ or }\ell|d_1,\\
                \frac{\widetilde{w}_2\#\Pic(\OO_{d_2})}
                    {w_2\#\Pic(\OO_{\widetilde{d_2}})}
                    \sum_{\tau_1}\#S_{1,m}(E(\tau_1)/\WW)
                    & \textup{if }\ell| f_2, \widetilde{d}_2,\\
                \frac{\widetilde{w}_2\#\Pic(\OO_{d_2})}
                    {2w_2\#\Pic(\OO_{\widetilde{d_2}})}
                    \sum_{\tau_1}\#S_{1,m}(E(\tau_1)/\WW)
                    & \textup{if }\ell| f_2, \ell\nmid\widetilde{d}_2.
    		\end{cases}
        \end{equation}

    {Now we turn to the computation of $e(\mu/\ell)$; recall our running assumption that either $\ell>2$ or $\ell\nmid\gcd(d_1,\widetilde{d}_2)$.  Note that since $d_1$ is fundamental at $\ell$, $L = L_{2,\mu}^{\textup{unr}}(\sqrt{d_1})$.  Moreover, if $\ell|d_2$, then $\sqrt{d_1}\in L_{2,\mu}^{\textup{unr}}$, and if $\ell\nmid f_2$, then the only ramification comes from that of $\Q(\sqrt{d_1}, \sqrt{d_2})/\Q.$

    Assume that $\ell|f_2$.  By local class field theory and~\cite[Theorem 9(2)]{SerreTate}, the ramification degree of $L/\Q_{\ell}^{\textup{unr}}(\sqrt{d_2})$ is  $\ell^{s_2}$ if $\ell|\widetilde{d_2}$, $\ell^{s_2 - 1}(\ell + 1)$ if $\ell$ in inert in $\Q(\sqrt{d_2})$, and $\ell^{s_2 - 1}(\ell - 1)$ if $\ell$ is split in $\Q(\sqrt{d_2})$.  By~\cite[Theorem 7.24]{Cox-PrimesOfTheForm}, this is equal to $\frac{\widetilde{w}_2}{w_2}\cdot\left(\#\Pic(\OO_{d_2})/\#\Pic(\OO_{\widetilde{d}_2})\right)$.  Thus, we have:}
    %
	\[
		e(\mu/\ell) =
			\begin{cases}
				1 & \textup{if }\ell\nmid d_1d_2,\\
				2 & \textup{if }\ell|d_1d_2, \ell\nmid f_2,\\
                \frac{2\widetilde{w}_2\#\Pic(\OO_{d_2})}
                    {w_2\#\Pic(\OO_{\widetilde{d_2}})}
				 &
				\textup{if }\ell|f_2, \widetilde{d}_2,\\
                \frac{\widetilde{w}_2\#\Pic(\OO_{d_2})}
                    {w_2\#\Pic(\OO_{\widetilde{d_2}})}
				 &
				\textup{if }\ell|f_2, \ell\nmid\widetilde{d}_2.
			\end{cases}
	\]

    Combining this with~\eqref{eqn:ValToEndo}, we obtain:
        \begin{equation}\label{eqn:ValToEndo2}
    		\frac{Cw_1}{2}v_{\ell}(F(m)) =
    		\begin{cases}
    			\frac1e\sum_{\tau_1}\sum_{n\geq1}\#S_{n,m}(E(\tau_1)/A)
                    & \textup{if }\ell\nmid f_2,\\
                 \sum_{\tau_1}\#S_{1,m}^{\Lie}(E(\tau_1)/A)
                    & \textup{if }\ell| f_2.
    		\end{cases}
        \end{equation}

    It remains to compute $\sum_{\tau_1}\#S_{n,m}(E(\tau_1)/A)$.
    \begin{lemma}
     Let $m=0$.  Then  $\#S_{n,0}(E/A) = 0$, unless $d_2 = d_1\ell^{2k}$, in which case $\#S_{n,0}(E/A) = 2$.
    \end{lemma}

    \begin{proof}
    	When $m = 0$, then $S_{n,m}(E/A)$ consists of elements in $\OO_{d_1}$ with trace $d_2$ and norm $\frac14(d_2^2 - d_2)$.  In addition, the order generated by these elements must be optimally embedded at $p\ne\ell$.  Thus  $\#S_{n,0}(E/A) = 0$, unless $d_2 = d_1\ell^{2k}$, in which case $\#S_{n,0}(E/A) = 2$.
    \end{proof}

    When $m\neq 0$ we will compute $\sum_{\tau_1}\#S_{n,m}(E(\tau_1)/A)$ by giving an explicit presentation for $\End_{\WW/\pi^n}(E)$(\S\ref{sec:QuaternionOrders}), and then use this presentation to relate the elements of $S_{n,m}(E/\WW)$ to integral invertible ideals of norm $m\ell^{-r}$, where $r = 2n - 1$ if $\ell$ is inert in $\Q(\sqrt{d_2})$ and $r = n$ otherwise(\S\ref{sec:Main}).  More precisely we will prove in \S\ref{sec:Main}:
    %
    %
    %
	\begin{thm}\label{thm:DeterminingSnm}
		Assume that $\ell\nmid f_1$ and that $m\neq 0$.  Then $\sum_{\tau_1}\#S_{n,m}(E(\tau_1)/\WW)$ is equal to an explicitly computable weighted sum of the number of certain invertible ideals of norm $\ell^{-r}m$, where $r = 2n - 1$ if $\ell\nmid d_1$ and $r = n$ otherwise.  If, in addition, $m$ and $f_1$ are relatively prime, then
		\[
			\sum_{\tau_1}\#S_{n,m}(E(\tau_1)/\WW) = \frac{Cw_1}{2}\rho(m)\frakA(\ell^{-r}m),
		\]
		where
		\begin{align*}
			\rho(m) = &
				\begin{cases}
					0 &\textup{if } (d_1, -m)_p = -1 \textup{ for }p|d_1, p \nmid f_1\ell,\\
					2^{\#\{p|(m, d_1) : p\nmid f_2
					\textup{ or } p= \ell\}} & \textup{otherwise}.
				\end{cases}
		\end{align*}
		and
		\begin{align*}
			\frakA(N) = & \#\left\{
				\begin{array}{ll}
					& \Norm(\frakb) = N, \frakb \textup{ invertible},\\
					\frakb\subseteq\OO_{d_1} : & p\nmid\frakb
					\textup{ for all }p |\textup{gcd}(N, f_2), p\nmid \ell d_1\\
					& \frakp^3\nmid\frakb\textup{ for all }
					\frakp|p|\textup{gcd}(N, f_2, d_1), p\ne\ell
				\end{array}
				\right\}.
		\end{align*}
	\end{thm}
	We will also prove in~\S\ref{sec:Main} that $\rho(m)\frakA(\ell^{-r} m)$ can be expressed as a product of local factors.  
	%
	 This completes the proof of Theorem~\ref{thm:main}.\qed
}	

\section{Relating $\Isom_{A/\mu^n}$ to $S_{n}^{\Lie}$}\label{sec:IsomToLie}

	We retain the notation from the previous section.  From now on, we assume that $\ell\nmid f_1$.  In this section, we prove
	\begin{prop}\label{prop:IsomToLie}
		Let $E$ be any elliptic curve over $A$ with good reduction such that $E\isom \OO_{d_1}$.  Let $d_2$ be a quadratic imaginary discriminant different from $d_1$.  Then
		\[
		\sum_{\tau_2}\#\Isom_{A/\mu^n}(E, E(\tau_2)) =
		\begin{cases}
			0 & \textup{if }\ell|f_2\textup{ and }n>1,\\
			\widetilde{w_2}\frac{\#\Pic(\OO_{d_2})}{\#\Pic(\OO_{\widetilde{d_2}})} \#S_{n}^{\Lie}(E) & \textup{otherwise}.
		\end{cases}
		\]
		where $e_2(\mu/\ell)$ is the ramification degree of $L_{\mu}^{\unr}$ over the completion of the ring class field of $\OO_{d_2}$ at the restriction of $\mu$.
	\end{prop}
	\begin{proof}
        Assume that $\ell|f_2$.  By~\cite{LT-FormalModuli} and by comparing ramification degrees of $\ell$ in $L_2$ and in $H_{\widetilde{d_2}}$, we see that $\Isom_{A/\mu^n}(E, E(\tau_2)) = \varnothing$ for all $\tau_2$ of discriminant $d_2$. Henceforth, we restrict to the case that $\ell\nmid f_2$ or $n = 1$.

		Let $E' = E(\tau_2)$ for some $\tau_2$. Assume that $\Isom_{A/\mu^n}(E, E')\neq\varnothing$, and let $g\in \Isom_{A/\mu^n}(E, E')$.  If $\theta'\in\End(E')$ is the unique element such that $c(\theta') = \delta$ in $\Lie(E')$, then $\theta'^g := g^{-1}\circ\theta\circ g\in\End(E\bmod\mu^n)$ has degree equal to $\frac14(d_2^2 - d_2)$, trace equal to $d_2$, and $\theta'^g = \delta$ in $\Lie(E \bmod \mu^n)$.  In addition, the order $\Z[\theta'^g]$ is $p$-optimally embedded in $\End(E\bmod\mu)$ for all $p\neq\ell$ by Proposition~\ref{prop:optimal_embedding}.  If $p = \ell$, then $\widetilde{\theta'} := \ell^{-s_2}\theta' + \frac12\widetilde{d}_2(1 - \ell^{s_2}) \in\End(E'\bmod\mu)$ and $c(\widetilde{\theta'}),$ $c(\widetilde{\theta'}^g)$  are equal to $\widetilde{\delta}$ in $\Lie(E'\bmod\mu)$, $\Lie(E\bmod\mu)$ respectively.  Therefore, we have a set map
		\begin{equation}\label{eq:IsomToLie}
            \bigcup_{[\tau_2]}\Isom_{A/\mu^n}(E, E(\tau_2)) \longrightarrow
            S_{n}^{\Lie}(E/A).
		\end{equation}
        {Further, this map is surjective by the Serre-Tate lifting theorem~\cite[Thm. 3.3]{Conrad-GZ}, the Grothendieck Existence theorem~\cite[Thm. 3.4]{Conrad-GZ}, and the existence of canonical and quasi-canonical liftings (see~\cite{Gross-CanonicalLifts} in the supersingular case and~\cite[Prop. 3.5]{Meusers-CanonicalLifts} in the ordinary case).  To complete the proof, we must determine the size of a fiber.


         Let $g\in \Isom_{A/\mu^n}(E,E')$.  Then the pairs $(E'' = E(\tau_2''), g''\in \Isom_{A/\mu^n}(E, E'')$ such that $\theta''^{g''} = \theta'^{g'}$ is in bijection with isomorphisms
         %
          $f\colon E'\bmod{\mu^n}\to E''\bmod{\mu^n}$ such that the diagram
         \[\xymatrix{
             E'\bmod{\mu^n} \ar[d]^f\ar[r]^{\theta'} & E'\bmod{\mu^n}\ar[d]^f\\
             E''\bmod{\mu^n} \ar[r]^{\theta''} & E''\bmod{\mu^n}\\
         }
         \]
         commutes.  In particular, if $\theta'^{g'} = \theta''^{g''}$ then the diagram commutes when $n = 1$ with $\theta'$ and $\theta''$ replaced with $\widetilde\theta'$ and $\widetilde\theta''$ respectively.  By~\cite[Prop. 2.7]{GZ-SingularModuli}, there exist $\#\Pic(\OO_{\widetilde{d_2}})$ distinct isomorphism classes of pairs $(E'\bmod \mu, \theta'\in\End_{A/\mu}(E'))$.  In particular, these isomorphisms classes correspond exactly to the isomorphism classes $(\widetilde E', \widetilde\theta'\in \End(\widetilde E'))$.  So the above diagram commutes for arbitrary $n$ if and only if it commutes for $n=1$.  (Recall if $\ell|f_2$, then we need only consider the case $n=1$.)

         Now it suffices to consider the case where $E' = E''$.  In that case, we are interested in determining the number of automorphisms of $E'\bmod{\mu^n}$ which commute with $\theta'$.  Since the centralizer of $\theta'$ is exactly $\Q(\theta')\cap\End_{A/\mu^n}(E')$, there are exactly $\widetilde{w_2}$ such automorphisms.  In summary, the fibers have cardinality
         \[
             {\#\Pic(\OO_{d_2})}\cdot\frac1{\#\Pic(\OO_{\widetilde{d_2}})}\cdot \widetilde{w_2},
         \]
         which completes the proof.
         }
\end{proof}

\section{Background: quadratic imaginary orders}
\label{sec:QuadImagOrders}\label{sec:Background}
	Let $\OO$ be an order in a quadratic imaginary field, and let $d$ be the discriminant of $\OO$.  Let $\fraka$ be an ideal in $\OO$.  If $\OO$ is not maximal, then we can not necessarily write $\fraka$ uniquely as a product of primes.  However, we can always write $\fraka$ uniquely as a product of primary ideals where no two ideals in the factorization are supported at the same prime.  Precisely, for any prime $\pp$, define $\fraka_{\pp} := \OO\cap\fraka\OO_{\pp}$.  Then $\fraka = \bigcap_{\pp}\fraka_{\pp}$, and since for any $2$ distinct primes $\pp, \qq$, $\fraka_{\pp}$ and $\fraka_{\qq}$ are co-maximal, we have that
	\[
		\fraka = \prod_{\pp}\fraka_{\pp}.
	\]
	(See~\cite[Prop 12.3]{Neukirch} for more details.)  If there is a unique prime $\pp\subseteq\OO$ lying over $p$, then we will often write $\fraka_{p}$ instead of $\fraka_{\pp}.$
	
	We will often be concerned with the special case where $\fraka = \frakD := \sqrt{d}\OO$.  If $p|d$ is odd, then for $a,b\in\OO$, the difference $a - b\in\frakD_p$ if and only if $\Tr(a)\equiv\Tr(b)\pmod{p^{v_p(d)}}$.  If $p = 2|d$, then $a - b\in\frakD_2$ if and only if $a_0\equiv b_0\pmod{2^{v_2(d) - 1}}$ and $a_1\equiv b_1\pmod{2},$ where $a = a_0 + a_1 \frac{d+\sqrt{d}}{2}$ and $b = b_0 + b_1\frac{d + \sqrt{d}}{2}$.
	
	\subsection{The Picard group}

		The Picard group of $\OO$, denoted $\Pic(\OO)$, is
the group of invertible fractional ideals modulo fractional principal ideals. It  is isomorphic to the \defi{form class group} $C(d)$, the group of classes of primitive positive definite forms of discriminant $d$~\cite[\S7]{Cox-PrimesOfTheForm}.  We will use this isomorphism to determine whether there exists an ideal in $2\Pic(\OO)$ of a certain norm.  For more information on genus theory, i.e. the study of $\Pic(\OO)/2\Pic(\OO)$, see~\cite{Cox-PrimesOfTheForm}.

		Let $p_1, \ldots, p_j$ be the {distinct} odd primes dividing $d$.  Define
		\[
			k = \begin{cases}
				j & \textup{ if }d \equiv 1 \pmod 4 \textup{ or }
					d \equiv 4 \pmod{16},\\
				j + 1 & \textup{ if }d \equiv 8\textup{ or } 12\pmod{16} \textup{ or }
					d \equiv 16 \pmod{32},\\
				j + 2 & \textup{ if }d \equiv 0 \pmod{32}.
			\end{cases}
		\]
		For $i = 1, \ldots, j$, we define $\chi_{p_i}(a) := \left(\frac{a}{p_i}\right)$ for $a$ coprime to $p_i$.  For $a$ odd, we also define $\chi_{-4}(a) := \left(-1 \right)^{\frac{a - 1}{2}}$, $\chi_8(a) := \left(-1\right)^{\frac{a^2 - 1}{8}}.$  Then we define $\Psi\colon \left(\Z/d\Z\right)^{\times}\to \{\pm1\}^{k}$ as follows.
		\[
			\Psi = \begin{cases}
				(\chi_{p_1}, \ldots, \chi_{p_j}) & \textup{ if }
					d \equiv 1 \pmod 4 \textup{ or }
					d \equiv 4 \pmod{16},\\
				(\chi_{p_1}, \ldots, \chi_{p_j}, \chi_{-4}) & \textup{ if }
					d \equiv 12 \pmod{16} \textup{ or }
					d \equiv 16 \pmod{32},\\
				(\chi_{p_1}, \ldots, \chi_{p_j}, \chi_8) & \textup{ if }
					d \equiv 8 \pmod{32},\\
				(\chi_{p_1}, \ldots, \chi_{p_j},\chi_{-4}\chi_8) & \textup{ if }
					d \equiv 24 \pmod{32},\\
				(\chi_{p_1}, \ldots,\chi_{p_j},\chi_{-4},\chi_8) & \textup{ if }
					d \equiv 0 \pmod{32}.\\
			\end{cases}
		\]
		For a prime $p$ that divides $d$, but does not divide the conductor $f$ of $\OO$, we define
		\[
			\Psi_{p} =
				\begin{cases}
					\chi_{p_i} & \textup{ if } p = p_i,\\
					\chi_{-4} & \textup{ if } p = 2
							 \textup{ and } d\equiv 12\pmod{16},\\
					\chi_8 & \textup{ if } p = 2
							   \textup{ and } d\equiv 8\pmod{32},\\
					\chi_{-4}\cdot\chi_8 & \textup{ if } p = 2
								   \textup{ and } d\equiv 24\pmod{32}.
				\end{cases}
		\]
		Let $\widehat{\Psi}_{p}$ be the projection of $\Psi$ on the components that are complementary to the one that appears in $\Psi_{p}$.

		{For $p\nmid f$}, we may extend $\Psi_{p}$ to integers divisible by $p$ by defining 
		$\Psi_p(n)$ to be the Hilbert symbol $(d, n)_p$, for any integer $n$ coprime to $f$.   Thus we can extend $\Psi$ to $\left(\Z/f\Z\right)^{\times}$, where $f$ denotes the conductor of $\OO$.
		
		This map $\Psi$ can be used to test when an ideal $\fraka$ {that is prime to $f$} is a square in the Picard group.
		\begin{thm}[\cite{Cox-PrimesOfTheForm}*{\S\S3\&7}]\label{thm:sqclasses}
			  For any positive integer $m$ prime to the conductor $f$ of $\OO_d$, there exists an invertible ideal $\fraka$ such that $\NN(\fraka) = m$ and $[\fraka]\in2\Pic(\OO_d)$ if and only if $m \in \ker \Psi$.
		\end{thm}
		\noindent From this theorem, we can easily obtain the following corollary.
		\begin{cor}\label{cor:genus}
			Let $\ell$ be a prime that divides $d$, but does not divide the conductor $f$.  Let $\fraka$ be an invertible integral ideal that is prime to the conductor.  Then $[\fraka]\in2\Pic(\OO)$ if and only if $\Norm(\fraka)\in\ker\widehat{\Psi}_{\ell}$.
		\end{cor}
		\begin{proof}
			Define $\psi\colon (\Z/f\Z)^{\times}\to \{\pm1\}$ to be such that for any positive integer $m$ that is coprime to $f$, $\psi(m) = 1$ if and only if there is an ideal in $\OO$ of norm $m$.  Using quadratic reciprocity, one can check that
			\[
				\psi(m) = \prod_{p|\widetilde{d}}\Psi_p(m).
			\]
			From this it is clear that $\Norm(\fraka)\in\ker\widehat{\Psi}_{\ell}$ if and only if $\Norm(\fraka)\in\ker{\Psi}$, which completes the proof.
		\end{proof}

        \subsection{{Genus class of ideals supported at the conductor}}
		Unfortunately, the map $\Psi$ cannot be extended to all integers while still retaining the properties described in Theorem~\ref{thm:sqclasses} and Corollary~\ref{cor:genus}.  This is because it is possible to have two invertible ideals $\fraka, \frakb\subseteq \OO_d$ with the same norm, such that $\fraka\frakb^{-1} \not\in2\Pic(\OO_d)$.  This can only occur when the ideals are not relatively prime to $f$.
		
		Let $\fraka$ be an integral invertible ideal that is supported at a single prime $p$ that divides the conductor, i.e. $\fraka_\frakq = \langle 1 \rangle$ for all $\frakq\nmid p$.  Let $\alpha\in\OO$ be a generator for $\fraka\OO_p$ such that $\textup{gcd}(\Norm(\alpha), f)$ is supported only at $p$.  Then $\fraka\sim\widetilde{\fraka}$ in $\Pic(\OO)$, where
		\[
			\widetilde{\fraka} := \OO_p\cap
			\bigcap_{\qq\nmid p}\left(\alpha\OO_\qq\right),
		\]
		and $\Norm(\widetilde{\fraka})$ is coprime to the conductor.  Thus, the genus of $\fraka$ is equal to $\Psi(\Norm(\widetilde{\fraka})).$  Since every ideal can be factored uniquely into comaximal primary ideals, this gives a method of computing the genus class of any ideal.
		
\section{Parametrizing endomorphism rings of supersingular elliptic curves}
\label{sec:DormanRepns}\label{sec:QuaternionOrders}

	Let $\ell$ be a fixed prime and let $\OO$ be a quadratic imaginary order of discriminant $d$ such that $\ell\nmid f:= \cond(d)$.  We assume that $\ell$ is not split in $\OO$.  Let $\WW$ be the ring of integers in $\Q_{\ell}^{\unr}(\sqrt{d})$, and write $\pi$ for the uniformizer.  By the theory of complex multiplication~\cite[\S10.3]{Lang-EllipticFunctions}, the isomorphism classes of elliptic curves that have CM by $\OO$ are in bijection with $\Pic(\OO)$, and every elliptic curve $E$ with CM by $\OO$ has a model defined over $\WW$.  Moreover, by~\cite[Cor. 1]{SerreTate}, we may assume that $E$ has good reduction.
	
	Fix a presentation $\BBl$ of the quaternion algebra ramified at $\ell$ and $\infty$, and fix an embedding $L := \textup{Frac}(\OO)\hookrightarrow \BBl$.  The goal of this section is to define, for every $[\fraka]\in\Pic(\OO)$, a maximal order $R(\fraka)\subset\BBl$ such that
	\begin{enumerate}
		\item $R(\fraka)\cap L = \OO$,
		\item The optimal embedding $\OO\hookrightarrow R(\fraka)$ is isomorphic to the embedding $\End(E(\fraka)) \hookrightarrow \End(E(\fraka)\bmod\pi)$, where $E(\fraka)$ is the elliptic curve with CM by $\OO$ that corresponds to $\fraka$, and
		\item $\frakb^{-1}R(\fraka)\frakb = R(\fraka\frakb)$.
	\end{enumerate}
	Since we will use these maximal orders in the next section to compute the sets $S_{n,m}(E(\fraka))$, we also want the orders $R(\fraka)$ to be fairly explicit.

	Our construction of these maximal orders $R(\fraka)$ generalizes the work of Gross-Zagier~\cite{GZ-SingularModuli} and Dorman~\cite{Dorman-Orders}, where they defined maximal orders with these properties under the assumption that $-d$ is prime~\cite{GZ-SingularModuli} or $d$ is squarefree~\cite{Dorman-Orders}.  We treat arbitrary discriminants $d$, correct errors and omissions in some proofs in~\cite{Dorman-Orders}, and treat the ramified case in detail, giving complete definitions and proofs.
	
	Note that Goren and the first author have given a different generalization of Dorman's work~\cite{GL} to higher dimensions, which works for CM fields $K$, characterizing superspecial orders in a quaternion algebra over the totally real field $K^+$ with an optimal embedding of $\calO_{K^+}$.  That work also corrects the proofs of~\cite{Dorman-Orders}, but in a slightly different way than we do here, and does not handle the ramified case or non-maximal orders.

	\subsubsection{Outline}
	In~\S\ref{sec:choosing-q}, we give an explicit presentation of $\BBl$ that we will work with throughout.  The construction of the maximal orders $R(\fraka)$ depends on whether $\ell$ is inert or ramified in $\OO$.  The inert case is discussed in detail in~\S\ref{subsec:inert}, and the construction in the ramified case is given in~\S\ref{subsec:ram}.  In these sections we also prove that our construction satisfies conditions $(1)$ and $(3)$. While the construction of $R(\fraka)$ is different in the ramified case, many of the proofs go through as in the inert case with minor modifications.  Because of this, in~\S\ref{subsec:ram}, we only explain the modifications and omit the rest of the proofs. In~\S\ref{subsec:CM}, we show that these constructions also satisfy property $(2)$.

	\subsection{Representations of quaternion algebra}\label{sec:choosing-q}

		Given a fixed embedding $\iota\colon L \hookrightarrow \BBl$, the quaternion algebra $\BBl$ can be written uniquely as $\iota(L) \oplus \iota(L)j$, where $j \in \BBl$ is such that $j\iota(\alpha)j^{-1} = \iota(\alphabar)$, for all $\alpha \in L$.  Thus $j^2$ defines a unique element in $\Q^{\times}/\NN(L^{\times})$.  From now on, we will represent $\BBl$ as a sub-algebra of $\Mat_2(L)$ as follows.
		\begin{equation}\label{eq:BBl}
			\BBl = \left\{ [\alpha:\beta] :=
				\begin{pmatrix}
					\alpha & \beta \\
					j^2 \betabar & \alphabar
				\end{pmatrix}
				: \alpha, \beta \in L\right\}.
		\end{equation}
Under this representation, $\iota\colon L\into \BBl, \quad \iota(\alpha) = [\alpha, 0].$

		If $\ell$ is unramified in $\OO$ then we may assume that $j^2 = -\ell q$, where $q$ is a prime such that $-\ell q \in \ker \Psi$ and $q\nmid d$.  If $\ell$ is ramified, then we may assume that $j^2 = -q$ where $-q\in\ker\widehat{\Psi}_{\ell}$, $-q \not\in \ker \Psi_{\ell}$ and $q\nmid d$.  (Recall that $\Psi$, $\widehat{\Psi}_{\ell}$ and $\Psi_{\ell}$ were defined in \S\ref{sec:Background}.)  In both cases, these conditions imply that $q$ is split in $\OO$.

	\subsection{The inert case}\label{subsec:inert}

		Let $\fraka\subseteq \OO$ be an integral invertible ideal such that $\textup{gcd}(f, \Norm(\fraka))= 1$.  Let $\frakq$ be a prime ideal of $\OO$ lying over $q$.  For any $\lambda\in\OO$ such that
		\begin{enumerate}
			\item $\lambda\frakq^{-1}\frakabar\fraka^{-1}\subseteq\OO$, and
			\item $\Norm(\lambda) \equiv -\ell q\pmod{d}$,
		\end{enumerate}
		we define
		\[
			R(\fraka, \lambda) := \left\{ [\alpha, \beta] :
			 	\alpha \in \frakD^{-1},
				\beta \in
				\frakq^{-1}\ell^{n - 1}\frakD^{-1}\frakabar\fraka^{-1},
				\alpha - \lambda\beta \in \OO \right\}.
		\]
		From this definition, it is clear that if $\lambda'$ satisfies $(1)$ and $(2)$ and $\lambda \equiv \lambda' \pmod {\frakD}$, then $R(\fraka, \lambda) = R(\fraka, \lambda')$.  We claim that, for any $\fraka$ and $\lambda$ satisfying conditions $(1)$ and $(2)$, $R(\fraka,\lambda)$ is a maximal order.
		
		\begin{remark}
			Although Dorman~\cite{Dorman-Orders} does not include condition $(1)$ in his definition, it is, in fact, necessary.  Without this assumption $R(\fraka, \lambda)$ is not closed under multiplication, even if $d$ is squarefree.  This was already remarked on in~\cite{GL}.
		\end{remark}
		\begin{remark}
			Write $\lambda = \lambda_0 + \lambda_1\frac{d + \sqrt{d}}{2}$.  If $d$ is odd, then the congruence class of $\lambda \bmod \frakD$ is determined by $\lambda_0\bmod{d}$.  In addition, the condition that $\Norm(\lambda) \equiv -\ell q\pmod{d}$ is equivalent to the condition that $\lambda_0^2\equiv -\ell q \pmod{d}$.  Therefore, if $d$ is odd, then we may think of $\lambda$ as an integer, instead of as an element of $\OO$.  This was the point of view taken in~\cite{GZ-SingularModuli, Dorman-Orders}.
		\end{remark}

		\begin{lemma}
			$R(\fraka, \lambda)$ is an order.
		\end{lemma}
		\begin{proof}
			We will show that $R(\fraka, \lambda)$ is closed under multiplication.  All other properties are easily checked.  Consider
			\[
				\left[\frac{a_1}{\sqrt{d}}, \frac{b_1}{\sqrt{d}}\right],
				\left[\frac{a_2}{\sqrt{d}}, \frac{b_2}{\sqrt{d}}\right]
				\in R(\fraka, \lambda).
			\]
			Their product is in $R(\fraka, \lambda)$ if and only if
			\begin{enumerate}
				\item $a_1a_2 + \ell q b_1\overline{b}_2 \in \frakD$
				\item $a_1b_2 - \overline{a}_2b_1 \in
						\frakD\frakq^{-1}\frakabar\fraka^{-1}$
				\item $a_1a_2 + \ell q b_1\overline{b}_2 - \lambda a_1b_2 +
						\lambda\overline{a}_2b_1 \in d\OO$
			\end{enumerate}

			\textbf{Claim 1:} Note that $a_1a_2 + \ell q b_1\overline{b}_2$ can be rewritten as
			\begin{equation}\label{eqn:alpha-cond}
			 	(a_1 - \lambda b_1)a_2 + \lambda b_1(a_2 - \lambda b_2)
				+ \lambda b_1(\lambda b_2- \overline{\lambda}\overline{b}_2)
				+ (\Norm(\lambda) + \ell q)b_1\overline{b}_2.
			\end{equation}
			Using the definition of $R(\fraka, \lambda)$ and the fact that for any $c \in \calO$, $(c - \overline{c}) \in \frakD$, one can easily check that~\eqref{eqn:alpha-cond} is in $\frakD$.

			\textbf{Claim 2:} We rewrite $a_1b_2 - \overline{a}_2b_1$ as $(a_1 - \lambda b_1)b_2 - (\overline{a}_2 - \overline{\lambda{b}}_2)b_1 + b_1(\lambda b_2 - \overline{\lambda b}_2).$  From this description, one can easily check that $a_1b_2 - \overline{a}_2b_1\in\frakD\frakq^{-1}\frakabar\fraka^{-1}$.

			\textbf{Claim 3:} Let $a_i' \in \OO$ be such that $a_i = \lambda b_i + \sqrt{d}a_i'$.  Then we can rewrite $a_1a_2 + \ell q b_1\overline{b}_2 - \lambda a_1b_2 + \lambda\overline{a}_2b_1$ as
			\begin{align}
				&\lambda^2b_1b_2 + \sqrt{d}\lambda(b_1a_2' + a_1'b_2) + da_1'a_2' + \ell qb_1\overline{b}_2
				- \lambda^2 b_1b_2 - \lambda\sqrt{d}a_1'b_2
				+ \Norm(\lambda)\overline{b}_2b_1 - \lambda\sqrt{d}b_1\overline{a'}_2\\
				&= da_1'a_2' + (\Norm(\lambda) + \ell q)b_1\overline{b}_2
				+ \lambda\sqrt{d}b_1(a_2' - \overline{a}_2').
				\label{eqn:lambda-cond}
			\end{align}
			Similar arguments as above show that~\eqref{eqn:lambda-cond} is in $d\OO$.
		\end{proof}
		\begin{lemma}
			The discriminant of $R(\fraka, \lambda)$ is $\ell^2$, and so $R(\fraka, \lambda)$ is a maximal order.
		\end{lemma}
		\begin{proof}
			To prove this lemma, we will use an auxiliary (non-maximal) order
			\[
				\Rtilde(\fraka) := \left\{[\alpha, \beta] :
				\alpha \in \OO, \beta \in \frakabar\fraka^{-1} \right\}.
			\]
			One can easily check that this is an order.  Let $\omega_1, \omega_2$ be a $\Z$-basis for $\frakabar\fraka^{-1}$.  Then
			\[
				[1,0],\; [(d + \sqrt{d})/2, 0],\; [0, \omega_1],\; [0, \omega_2]
			\]
			is a $\Z$-basis for $\Rtilde(\fraka)$.  Using these basis elements, one can show that $\disc(\Rtilde(\fraka)) = (\ell q d)^2$.

			We claim that $\Rtilde(\fraka)$ is related to $R(\fraka, \lambda)$ by the following exact sequence,
			\[
				0 \to \Rtilde(\fraka) \to R(\fraka, \lambda)
				\stackrel{[\alpha,\beta]\mapsto\beta}{\longrightarrow}
				\frakq^{-1}\frakD^{-1}\frakabar\fraka^{-1}/
				\frakabar\fraka^{-1} \to 0.
			\]
			This exact sequence implies that $\left[R(\fraka, \lambda):\Rtilde(\fraka)\right] = q d$, and thus the lemma follows.

		\noindent\textit{Proof of claim:}
			Exactness on the left is straightforward.  Exactness on the right holds since for any $\beta\in\frakq^{-1}\frakD^{-1}\frakabar\fraka$, $[\lambda\beta, \beta]\in R(\fraka, \lambda)$.  The order $\Rtilde(\fraka)$ is clearly in $\ker([\alpha, \beta] \mapsto \beta)$, and the reverse containment follows since $\beta \in \frakabar\fraka^{-1}$ implies that $\lambda_{\fraka}\beta$, and hence $\alpha$, is in $\OO$.
		\end{proof}

		Given an ideal $\fraka$, we will now construct a $\lambda = \lambda_{\fraka}$ satisfying conditions $(1)$ and $(2)$.  Since we want our orders $R(\fraka) := R(\fraka, \lambda_{\fraka})$ to satisfy
		\[
			R(\fraka)\frakb = \frakb R(\fraka\frakb),
		\]
		the relationship between $\lambda_{\fraka}$ and $\lambda_{\fraka\frakb}$ will be quite important.  In fact, the relation
		\[
			R(\OO)\fraka = \fraka R(\fraka)
		\]
		shows that $R(\fraka)$ is determined from $R(\OO)$ and so $\lambda_{\fraka}\bmod{\frakD}$ is determined by $\lambda_{\OO}\bmod{\frakD}$.
		
		\subsubsection{Defining $\lambda_{\fraka}$}
		For all invertible ramified primes $p$, fix two elements $\lambda^{(p)}, \lambdatilde^{(p)}\in \OO$ with norm congruent to $-\ell q \bmod{p^{v(d)}}$ such that $\lambda^{(p)}\not\equiv \lambdatilde^{(p)}\pmod{\frakD_p}$.  For all non-invertible ramified primes $p$, fix $\lambda^{(p)}\in\OO$ such that $\Norm(\lambda^{(p)}) \equiv -\ell q \pmod{p^{v(d)}}$.

		For any prime ideal $\pp$ of $\OO$ that is prime to $\frakD$, let $M(\pp)$ denote a fixed integer that is divisible by $\Norm(\pp)$ and congruent to $1\pmod{d}$.  For any product of invertible ramified primes $\frakb_d := \prod_{\substack{\pp\textup{ invertible }\\\pp\mid\frakD}} \pp^{e_{\pp}}$, we write $\lambda_{\frakb_d}$ for any element in $\OO$ such that
		\[
			\lambda_{\frakb_d} \bmod{\frakD_p}\equiv
				\begin{cases}
					\lambda^{(p)} & \textup{if }e_p\equiv 0 \pmod 2\\
					\lambdatilde^{(p)} & \textup{if }e_p\equiv 1 \pmod 2
				\end{cases}
		\]
		for all invertible ramified primes $\pp$ and $\lambda_{\frakb_d} \equiv \lambda^{(p)} \pmod{\frakD_p}$ for all non-invertible primes.  These conditions imply that $\lambda_{\frakb_d}$ is well-defined modulo $\frakD$.
		
		Let $\fraka$ be an invertible integral ideal $\OO$ such that $\textup{gcd}(\Norm(\fraka), f) = 1$.  Then we may factor $\fraka$ as $\fraka'\fraka_d$, where $\fraka'$ is prime to the discriminant and $\fraka_d$ is supported only on invertible ramified primes.  We define $\lambda_{\fraka} := \left(\prod_{\pp|\fraka'}M(\pp)^{v_{\pp}(\fraka')}\right)M(\frakq)\lambda_{\fraka_d}$. Note that it follows from this definition that  $\lambda_{\fraka}$ is well-defined modulo $\frakD$ and, importantly, that $\lambda_{\fraka}$  satisfies $\lambda_{\fraka}\frakq^{-1}\frakabar\fraka^{-1} \subset \OO$ and $\Norm(\lambda_{\fraka}) \equiv -\ell q \mod d$.

		\begin{lemma}\label{lem:conj-inert}
			Let $\fraka, \frakb$ be two invertible ideals in $\OO$ that are prime to the conductor.  Assume that $\fraka$ and $\fraka\frakb$ are both integral.  Then
			\[
				R(\fraka, \lambda_{\fraka})\frakb = \frakb R(\fraka\frakb, \lambda_{\fraka\frakb}).
			\]
		\end{lemma}
		\begin{proof}
			We will show that $R(\fraka, \lambda_{\fraka})\frakb \subseteq \frakb R(\fraka\frakb, \lambda_{\fraka\frakb})$. The reverse containment then follows by letting $\fraka = \fraka\frakb$ and $\frakb = \frakb^{-1}$.  Note that if $\frakb = \frakb_1\frakb_2^{-1}$ with $\frakb_i$ integral, we may rewrite $\frakb$ as $\frakb_1\overline{\frakb_2}\Norm(\frakb_2)^{-1}$.  Since $\Norm(\frakb_2)^{-1}$ is in the center of $\BBl$ and $R(\fraka, \lambda_\fraka) = R(N\fraka, \lambda_{N\fraka})$ for any integer $N$, we may reduce to the case where $\frakb$ is integral.  We will also assume that $\frakb = \frakp$ is prime; the general result follows from multiplicativity.

			We can write
			\[
				\frakp R(\fraka\frakp, \lambda_{\fraka\frakp}) = \left\{
				[\tilde\alpha, \tilde\beta] :
                \tilde\alpha \in \frakp\frakD^{-1},
				\tilde\beta \in
                    \frakq^{-1}\frakD^{-1}\frakabar\overline{\frakp}
				\fraka^{-1},
				\tilde\alpha - \lambda_{\fraka\frakp}\tilde\beta
                    \in \frakp\right\}.
			\]
			Take $[\alpha, \beta]\in R(\fraka, \lambda_{\fraka})$ and $\gamma \in \frakp$; the product equals $[\gamma\alpha, \overline{\gamma}\beta]$.  One can easily see that $\gamma\alpha \in \frakp\frakD^{-1}$ and that $\overline{\gamma}\beta \in \frakq^{-1}\frakD^{-1}\frakabar\overline{\frakp}\fraka^{-1}$.  It  remains to show that $\gamma\alpha - \lambda_{\fraka\frakp}\overline{\gamma}\beta \in \frakp$.

			Consider the case where $\frakp$ is unramified.  Since {$\lambda_{\fraka\frakp} = M(\frakp)\lambda_{\fraka}$},
we may rewrite $\gamma\alpha - \lambda_{\fraka\frakp}\overline{\gamma}\beta$ as
			\[
				(\gamma - \overline{\gamma}M(\frakp))\alpha + \overline{\gamma}M(\frakp)(\alpha - \lambda_{\fraka}\beta).
			\]
			Since $p|M(\pp)$ and $M(\frakp)\equiv 1\pmod{d}$, one can easily check that $(\gamma - \overline{\gamma}M(\pp))\alpha$ and $\overline{\gamma}M(\pp)(\alpha - \lambda_{\fraka}\beta)$ are in $\frakp.$

			Now consider the case where $\frakp$ is ramified.  Then $\lambda_{\fraka} - \lambda_{\fraka\frakp} \in\frakD_{\pp'}$ for all ramified $\pp'\neq\frakp$, so it is clear that $v_{\pp'}(\gamma\alpha - \lambda_{\fraka\frakp}\overline{\gamma}\beta)\geq0$ for all $\pp'\neq\pp$.  It remains to show that $v_{\frakp}(\gamma\alpha - \overline{\gamma}\lambda_{\fraka\frakp}\beta)\geq 1 $.
			
			First consider the case when $\frakp|2$.  Rewrite $\gamma\alpha - \lambda_{\fraka\frakp}\overline{\gamma}\beta$ as
	\begin{equation}\label{eq:lambda_ap-inert}			
	\gamma(\alpha - \lambda_{\fraka}\beta) + \gamma(\lambda_{\fraka} - \lambda_{\fraka\frakp})\beta + \lambda_{\fraka\frakp}\beta(\gamma - \overline{\gamma}).
	\end{equation}
From the definition of $\lambda_{\fraka}$, we see that $v_{\frakp}(\lambda_{\fraka} - \lambda_{\fraka\frakp}) = v_{\frakp}(\frakD) - 1$.  Thus $\gamma(\lambda_{\fraka} - \lambda_{\fraka\frakp})\beta$ is a $\pp$-adic unit if  $v_{\frakp}(\beta) = - v_{\frakp}(\frakD)$ and $v_{\frakp}(\gamma) = 1$, and in $\frakp$ otherwise.  Moreover, the same characterization holds for $\lambda_{\fraka\frakp}\beta(\gamma - \overline{\gamma})$.  Since $\#\OO/\pp = 2$, a sum of two $\frakp$-adic units is in $\frakp$.  Hence, we conclude that $v_{\frakp}(\gamma\alpha - \overline{\gamma}\lambda_{\fraka\frakp}\beta)\geq 1 $.
			
			If $p$ is odd, then $v_{\frakp}(\beta) \geq -1$.  Combining the last two terms in  equation~\ref{eq:lambda_ap-inert},
we see that we need to prove that
$v_{\frakp}((\gamma\lambda_{\fraka} - \overline{\gamma}\lambda_{\fraka\frakp})\beta)\geq 1$,
so it suffices to prove that
$v_{\frakp}(\gamma\lambda_{\fraka} - \overline{\gamma}\lambda_{\fraka\frakp})\geq 2$.
			From the definition of $\lambda_{\fraka}$, we see that $\lambda_{\fraka} + \lambda_{\fraka\frakp}\in\frakp,$ and since $\gamma\in\frakp$, $\gamma+\overline{\gamma}\in \pp^2$.  Thus $\gamma(\lambda_{\fraka} + \lambda_{\fraka\frakp}) - (\gamma + \overline{\gamma})\lambda_{\fraka\frakp} = \gamma\lambda_{\fraka} - \overline{\gamma}\lambda_{\fraka\frakp}$ is in $\frakp^2$.
		\end{proof}

	\subsection{The ramified case}\label{subsec:ram}
	Let $\fraka\subseteq \OO$ be an integral invertible ideal such that $\textup{gcd}(f, \Norm(\fraka))= 1$.  Let $\frakq$ be a prime ideal of $\OO$ lying over $q$.  For any $\lambda\in\OO$ such that
	\begin{enumerate}
		\item $\lambda\frakq^{-1}\frakabar\fraka^{-1}\subseteq\OO$, and
		\item $\Norm(\lambda) \equiv - q\pmod{d/\ell}$,
	\end{enumerate}
	we define
	\[
		R(\fraka, \lambda) := \left\{ [\alpha, \beta] :
		 	\alpha \in \frakl\frakD^{-1},
			\beta \in \frakq^{-1}\frakl\frakD^{-1}\frakabar\fraka^{-1},
			\alpha - \lambda\beta \in \OO \right\}.
	\]
	From this definition, it is clear that if $\lambda'$ satisfies $(1)$ and $(2)$ and $\lambda \equiv \lambda' \pmod {\frakD\frakl^{-1}}$, then $R(\fraka, \lambda) = R(\fraka, \lambda')$.  We claim that, for any $\fraka$ and $\lambda$, $R(\fraka,\lambda)$ is a maximal order.
		\begin{lemma}
			$R(\fraka, \lambda)$ is an order.
		\end{lemma}
		\begin{proof}
			We will show that $R(\fraka, \lambda)$ is closed under multiplication.  All other properties are easily checked.  Consider
			\[
				\left[\frac{a_1}{\sqrt{d}}, \frac{b_1}{\sqrt{d}}\right],
				\left[\frac{a_2}{\sqrt{d}}, \frac{b_2}{\sqrt{d}}\right]
				\in R(\fraka, \lambda).
			\]
			Their product is in $R(\fraka, \lambda)$ if and only if
			\begin{enumerate}
				\item $a_1a_2 +  q b_1\overline{b}_2 \in \frakl\frakD$
				\item $a_1b_2 - \overline{a}_2b_1 \in
						\frakl\frakD\frakq^{-1}\frakabar\fraka^{-1}$
				\item $a_1a_2 + q b_1\overline{b}_2 - \lambda a_1b_2 +
						\lambda\overline{a}_2b_1 \in d\OO$
			\end{enumerate}
			The proof of these claims goes through exactly as in the inert case, after replacing every $q$ in the inert case with $q/\ell$, and after noting that $a_i, \overline{a}_i\in\frakl$ and $b_i, \overline{b}_i\in\frakl\frakq^{-1}\frakabar\fraka^{-1}$.
		\end{proof}
		\begin{lemma}
			The discriminant of $R(\fraka, \lambda)$ is $\ell^2$, and so $R(\fraka, \lambda)$ is a maximal order.
		\end{lemma}
		\begin{proof}
			This proof is exactly the same as in the inert case after replacing $q$ with $q/\ell$.
		\end{proof}

		\subsubsection{Defining $\lambda_\fraka$}
		For all invertible ramified primes $p$, fix two elements $\lambda^{(p)}, \lambdatilde^{(p)}\in \OO$ with norm congruent to $- q \bmod{p^{v(d/\ell)}}$ such that $\lambda^{(p)}\not\equiv \lambdatilde^{(p)}\pmod{\frakD_p}$.  If $p =\ell$ and $\ell \ne 2$, then in addition we assume that $\lambda_{\ell,0} = -\lambda_{\ell, 1}$.For all non-invertible ramified primes $p$, fix $\lambda^{(p)}\in\OO$ such that $\Norm(\lambda^{(p)}) \equiv -q \pmod{p^{v(d/\ell)}}$.

		For any prime ideal $\pp$ of $\OO$ that is coprime to $\frakD$, let $M(\pp)$ denote a fixed integer that is divisible by $\Norm(\pp)$ and congruent to $1\pmod{d}$.  For any product of invertible ramified primes $\frakb_d := \prod_{\substack{\pp\textup{ invertible }\\\pp\mid\frakD}} \pp^{e_{\pp}}$, we write $\lambda_{\frakb_d}$ for any element in $\OO$ such that
		\[
			\lambda_{\frakb_d} \bmod{\frakD_p}\equiv
				\begin{cases}
					\lambda^{(p)} & \textup{if }e_p\equiv 0 \pmod 2,\\
					\lambdatilde^{(p)} & \textup{if }e_p\equiv 1 \pmod 2.
				\end{cases}
		\]
		for all invertible ramified primes $\pp$ and $\lambda_{\frakb_d} \equiv \lambda^{(p)} \pmod{\frakD_p}$ for all non-invertible primes.  These conditions imply that $\lambda_{\frakb_d}$ is well-defined modulo $\frakD$.
		
		Let $\fraka$ be an invertible integral ideal $\OO$ such that $(\Norm(\fraka), f) = 1$.  Then we may factor $\fraka$ as $\fraka'\fraka_d$, where $\fraka'$ is coprime to the discriminant and $\fraka_d$ is supported only on invertible ramified primes.  We define $\lambda_{\fraka} := \left(\prod_{\pp|\fraka'}M(\pp)^{v_{\pp}(\fraka')}\right)M(\frakq)\lambda_{\fraka_d}$.  Note that $\lambda_{\fraka}$ is well-defined modulo $\frakD$ and that $\lambda_{\fraka}$ satisfies $\lambda_{\fraka}\frakq^{-1}\frakabar\fraka^{-1} \subset \OO$ and $\Norm(\lambda_{\fraka}) \equiv - q \mod d/\ell$.

		\begin{remark}
			Since $\lambda_{\fraka}\equiv\lambda_{\fraka\frakl}\pmod{\frakD\frakl^{-1}}$ for any integral invertible ideal $\fraka$, the corresponding orders $R(\fraka)$, $R(\fraka\frakl)$ are equal.  This is not surprising, since $E(\fraka) \isom E(\fraka\frakl)$ modulo $\pi$.
		\end{remark}

		\begin{lemma}\label{lem:conj-ram}
			Let $\fraka, \frakb$ be two invertible ideals in $\OO$ that are coprime to the conductor.  We assume that $\fraka$ and $\fraka\frakb$ are integral.  Then
			\[
				R(\fraka, \lambda_{\fraka})\frakb = \frakb R(\fraka\frakb, \lambda_{\fraka\frakb}).
			\]
		\end{lemma}
		\begin{proof}
			As in the proof of Lemma~\ref{lem:conj-inert}, it suffices to prove that $R(\fraka, \lambda_{\fraka})\frakb\subset\frakb R(\fraka\frakb, \lambda_{\fraka\frakb})$ and we can reduce to the case where $\frakb$ is prime.
			
			We can write
			\[
				\pp R(\fraka\pp, \lambda_{\fraka\pp}) =
                \left\{[\tilde\alpha,\tilde\beta]:
				\tilde\alpha \in \pp\frakl\frakD^{-1},
				\tilde\beta \in
                \frakq^{-1}\frakl\frakD^{-1}\frakabar\ppbar\fraka^{-1},
				\tilde\alpha - \lambda_{\fraka\pp}\tilde\beta\in\pp\right\}.
			\]
			Take $[\alpha,\beta]\in R(\fraka, \lambda_{\fraka})$ and $\gamma \in \frakp$: the product equals $[\gamma\alpha , \overline{\gamma}\beta]$.  One can easily see that $\gamma\alpha \in \pp\frakl\frakD^{-1}$ and that $\overline{\gamma}\beta\in \frakq^{-1}\frakl\frakD^{-1}\frakabar\ppbar\fraka^{-1}$.  It remains to show that $\gamma\alpha - \lambda_{\fraka\pp}\overline{\gamma}\beta \in \pp$.
			
			We will focus on the case where $\frakb = \frakl$; if $\frakb \ne \frakl$, then the proof is exactly as in Lemma~\ref{lem:conj-inert}.  We can rewrite $\gamma\alpha - \lambda_{\fraka\frakl}\overline{\gamma}\beta$ as
			\[
				\gamma(\alpha - \lambda_{\fraka}\beta) +
				\gamma\beta(\lambda_{\fraka} - \lambda_{\fraka\frakl}) +
				\lambda_{\fraka\frakl}\beta(\gamma - \overline{\gamma}).
			\]
			It is straightforward to see that the first and third terms are in $\frakl$.  The second term is in $\frakl$ since $v_{\frakl}(\lambda_{\fraka} - \lambda_{\fraka\frakl}) = v_{\frakl}(\frakD) - 1$.  This completes the proof.
		\end{proof}

	\subsection{Elliptic curves with complex multiplication}\label{subsec:CM}

		\begin{lemma}\label{lem:formofmaxlorders}
			Let $R$ be a maximal order of $\BBl$ such that $R\cap L = \OO$, where the intersection takes place using the embedding of $\BBl \subset \Mat_2(L)$ given in~\eqref{eq:BBl}.  Then there is an integral invertible ideal $\fraka \subseteq \OO$ coprime to the conductor such that $R$ is conjugate to $R(\fraka, \lambda_{\fraka})$ by an element of $L^{\times}$.
		\end{lemma}
		\begin{proof}
			Since $R$ is a maximal order in $\BBl$, $\OO$ must be maximal at $\ell$~\cite[Chap. 2, Lemma 1.5]{Vigneras}.  Therefore we can define $R(\OO, \lambda_{\OO})$ as in \S\ref{subsec:inert} or \S\ref{subsec:ram} depending on whether $\ell$ is inert or ramified in $\OO$.  Since $R(\OO)$ and $R$ both have $\OO$ optimally embedded, by~\cite[Thm. 4, p.118]{Eichler}, there exists an invertible ideal $\fraka'$ such that $R = \fraka'^{-1}R(\OO)\fraka'$.  We may write $\fraka'$ as $\theta\fraka$, where $\theta \in L^{\times}$ and $\fraka$ is integral and coprime to the conductor.  Thus, Lemmas~\ref{lem:conj-inert} and~\ref{lem:conj-ram} show that $R = \theta^{-1}R(\fraka, \lambda_{\fraka})\theta$.
			
		\end{proof}
	
		Fix an element $[\tau^{(0)}]$ of discriminant $d$, and let $E = E(\tau^{(0)})$ be an elliptic curve over $\WW$ with $j(E) = j(\tau^{(0)})$ and good reduction at $\pi$.  Then we have an optimal embedding of $\OO\isom\End(E)$ into $\End_{\WW/\pi}(E)$, a maximal order in $\BBl$.  Thus, by Lemma~\ref{lem:formofmaxlorders}, there is an element $[\fraka_0]\in\Pic(\OO)$ such that the pair
		\[
			\left(\End(E\bmod\frakl),
			\iota\colon \OO = \End(E)\hookrightarrow \End(E\bmod\frakl)\right)
		\]
		is conjugate to $R(\fraka_0)$ with the diagonal embedding $\OO\hookrightarrow R(\fraka_0)$.  Now let $\sigma\in\Gal(H/L)$ and consider the pair
		\[
			\left(\End(E^\sigma\bmod\frakl),
			\iota\colon \OO\hookrightarrow \End(E^\sigma\bmod\frakl)\right).
		\]
		By class field theory, $\Gal(H/L)\isom \Pic(\OO)$; let $\fraka = \fraka_\sigma$ be an invertible ideal that corresponds to $\sigma$; note that $\fraka$ is unique as an element of $\Pic(\OO)$.  We assume that $\fraka$ is integral and coprime to the conductor.  Since $\Hom(E^{\sigma}, E)$ is isomorphic to $\fraka$ as a left $\End(E)$-module, we have $\End(E^\sigma\bmod\frakl)=\fraka\End(E)\fraka^{-1}$\cite[Chap. XIII]{CasselsFroehlich}.  Thus, by Lemmas~\ref{lem:conj-inert} and~\ref{lem:conj-ram}, the pair corresponding to $E^{\sigma}$ is conjugate to $R(\fraka_0\frakabar)$.
		
		Now define
		\begin{align*}
				R_n(\fraka) := &
				\left\{ [\alpha, \beta] :
				 	\alpha \in \frakD^{-1},
					\beta \in
					\frakq^{-1}\ell^{n - 1}\frakD^{-1}\frakabar\fraka^{-1},
					\alpha - \lambda_{\fraka}\beta \in \OO \right\},
					\quad\textup{if }\ell\nmid d,\\
				R_n(\fraka) := &
				\left\{ [\alpha, \beta] :
				 	\alpha \in \frakD^{-1},
					\beta \in
					\frakq^{-1}\frakl^n\frakD^{-1}\frakabar\fraka^{-1},
					\alpha - \lambda_{\fraka}\beta \in \OO \right\},
					\quad\quad\textup{if }\ell| d.\\
		\end{align*}
		One can easily check that $R_1(\fraka) = R(\fraka)$, that $\bigcap_nR_n(\fraka) = \OO$ and that
			\begin{equation}\label{eq:Rnprop-inert}
				R_n(\fraka) =
					\begin{cases}
						\OO + \ell^{n - 1}R_1(\fraka)
				 & \textup{if }\ell\nmid d, \\
				\OO + \frakl^{n - 1}R_1(\fraka), &\textup{if }\ell|d.
			\end{cases}
			\end{equation}
		Then by~\cite[Prop. 3.3]{Gross-CanonicalLifts}, $\End_{\WW/\pi^n}(E^\sigma) \isom R_n(\fraka_0\frakabar)$	 .

\section{Proof of Theorem~\ref{thm:DeterminingSnm}}\label{sec:Main}
\label{sec:CountingEndomorphisms}

	Retain the notation from~\S\ref{subsec:DetailedOutline} and let  $S_{n,m}(E)$ denote $S_{n,m}(E/\WW)$.  In this section, we prove Theorem~\ref{thm:DeterminingSnm}, which we restate here for the reader's convenience.

	\begin{theorem_nonum}
		Assume that $\ell\nmid f_1$ and that $m\neq 0$.  Then $\sum_{\tau_1}\#S_{n,m}(E(\tau_1)/\WW)$ is equal to an explicitly computable weighted sum of the number of certain invertible ideals of norm $\ell^{-r}m$, where $r = 2n - 1$ if $\ell\nmid d_1$ and $r = n$ otherwise.  If, in addition, $m$ and $f_1$ are relatively prime, then
		\[
			\sum_{\tau_1}\#S_{n,m}(E(\tau_1)/\WW) = \frac{Cw_1}{2}\rho(m)\frakA(\ell^{-r}m),
		\]
		where $C = 1$ if $4m = d_1d_2$ and $C=2$ otherwise,
		\begin{align*}
			\rho(m) = &
				\begin{cases}
					0 &\textup{if } (d_1, -m)_p = -1 \textup{ for }p|d_1, p \nmid f_1\ell,\\
					2^{\#\{p|(m, d_1) : p\nmid f_2
					\textup{ or } p= \ell\}} & \textup{otherwise}.
				\end{cases}
		\end{align*}
		and
		\begin{align*}
			\frakA(N) = & \#\left\{
				\begin{array}{ll}
					& \Norm(\frakb) = N, \frakb \textup{ invertible},\\
					\frakb\subseteq\OO_{d_1} : & p\nmid\frakb
					\textup{ for all }p |\textup{gcd}(N, f_2), p\nmid \ell d_1\\
					& \frakp^3\nmid\frakb\textup{ for all }
					\frakp|p|\textup{gcd}(N, f_2, d_1), p\ne\ell
				\end{array}
				\right\}.
		\end{align*}
	\end{theorem_nonum}
\begin{remark}	
In the general case, i.e. if $m$ is not coprime to $f_1$, $\sum_{\tau_1}\#S_{n,m}(E(\tau_1))$ is still computable; in fact, the proof provides an algorithm.
\end{remark}
	\begin{proof}
		If $\ell$ is split in $\OO_{d_1}$, then the Hilbert symbol $(d_1, -m)_{\ell} = 1$.  Therefore, by the proof of Theorem~\ref{thm:DefnOfFm}, $\sum_{\tau_1}\#S_{n,m}(E(\tau_1)) = 0$, and so each invertible ideal of norm $\ell^{-r}m$ is said to have weight $0$.  Now restrict to the case where $m$ and $f_1$ are relatively prime.  Since $(d_1, -m)_{\infty} = -1$, by class field theory, there exists a finite prime $p\ne \ell$ such that $(d_1, -m)_p = -1$.  If $p$ is inert in $\Q(\sqrt{d_1})$ then this implies that $v_p(m) \equiv 1 \pmod 2$ and thus $\frakA(m\ell^{-r}) = 0$ for all $r$.  If $p$ is ramified in $\Q(\sqrt{d_1})$ then $\rho(m) = 0$.  So we conclude that if $\ell$ is split in $\OO_{d_1}$ then both sides are $0$.
		
		From now on we assume that the prime $\ell$ does not split in $\OO_{d_1}$.  Therefore, by~\S\ref{subsec:CM}, for every $[\tau_1]$ of discriminant $d_1$ there exists an integral invertible ideal $\fraka = \fraka_{\tau_1}$, such that $\End_{\WW/\pi^n}(E(\tau_1)) \isom R_n(\fraka)$.  Furthermore, the elements $\fraka_{\tau_1}$ can be chosen in a way that is compatible with the Galois action; we will assume that this is the case.  Since $R_n(\fraka) \subset \BBl$, henceforth $\Tr$, $\Norm$, and $\textup{disc}$ will refer to the reduced trace, reduced norm, and reduced discriminant, respectively, in the quaternion algebra.  With the identification of $\End_{\WW/\pi^n}(E(\tau_1))$ with $R_n(\fraka)$, $S_{n,m}(E(\tau_1))$ consists of elements $[\alpha, \beta]$ in $ R_n(\fraka)$ such that
		\[
			\Tr([\alpha, \beta]) = d_2, \;
			\Norm([\alpha, \beta]) = \frac14(d_2^2 - d_2), \;
			\textup{disc}(\OO_{d_1}\oplus\OO_{d_1}[\alpha,\beta]) = m^2,
		\]
		and such that $\Z\oplus\Z[\alpha,\beta]$ in $\left(\Q\oplus\Q[\alpha,\beta]\right)\cap R(\fraka)$ has index a power of $\ell$.  Let $t$ denote the trace of $\frac12(d_1 + \sqrt{d_1})[\alpha,\beta]^{\vee}$.  Then a calculation shows that $4m = d_1d_2 - (d_1d_2 - 2t)^2$, or equivalently that $t = \frac12\left(d_1d_2\pm\sqrt{d_1d_2 - 4m}\right).$  The values of $\Tr([\alpha,\beta])$ and $\Tr\left(\frac12(d_1 + \sqrt{d_1})[\alpha,\beta]^{\vee}\right)$ imply that $\alpha = \frac{(d_1d_2 - 2t) + d_2\sqrt{d_1}}{2\sqrt{d_1}}$, so $\alpha$ is uniquely determined by $t$.  In turn, $t$ is uniquely determined by $m$ if $4m = d_1d_2$, and otherwise $t$ is determined by a choice of sign.  Define
		\[
			S_{n,m}^t(E) := \left\{[\alpha,\beta]\in S_{n,m}(E) : \Tr\left(\frac12(d_1 + \sqrt{d_1})[\alpha,\beta]^{\vee}\right) = t\right\},
		\]
		where $t$ and $m$ satisfy $4m = d_1d_2 - (d_1d_2 - 2t)^2$.  We will prove that $\sum_{\tau_1}\#S_{n,m}^t(E(\tau_1))$ is a weighted sum of integral ideals of norm $m/\ell^r$ and if $\textup{gcd}(m, f_1) = 1$ then $\sum_{\tau_1}\#S_{n,m}^t(E(\tau_1))=  \frac{w_1}{2}\rho(m)\frakA(\ell^{-r}m)$. { The theorem follows immediately from these two claims and the above discussion.
}		
		
		\begin{prop}\label{prop:obtainingb}
			Let $[\alpha,\beta]\in R_n(\fraka)$ as above, and define $\frakb := \beta\frakr^{-1}\frakD\frakabar^{-1}\fraka$, where $\frakr = \ell^{n - 1}\frakq^{-1}$ if $\ell$ is inert in $\OO_{d_1}$ and $\frakr = \frakl^n\frakq^{-1}$ if $\ell$ is ramified.  Then $\frakb$ is an integral invertible ideal of $\OO_{d_1}$ with the following properties
			\begin{align}\label{eq:cond-on-b1}
				\Norm(\frakb) &= m\ell^{-r},\\
				\label{eq:cond-on-b2} \frakb &\sim\frakr\pmod{2\Pic(\OO_{d_1})}.
			\end{align}
		\end{prop}
        \noindent (Recall that $r= 2n-1$ is $\ell\nmid d_1$ and $r=n$ otherwise.)
		\begin{proof}
			By the definition of $R_n(\fraka)$, we see that $\beta \in  \frakr\frakD^{-1}\frakabar\fraka^{-1}$ and since $\Norm([\alpha,\beta]) = \frac14(d_2^2 - d_2)$ we deduce that $\beta$ satisfies
		\[
			\beta\betabar =
				\begin{cases}
					\frac1{\ell q}
					\left( \frac14(d_2^2 - d_2) - \alpha\alphabar \right)
					& \textup{if }\ell\textup{ is inert in }\OO_{d_1},\\
					\frac1{q}
					\left( \frac14(d_2^2 - d_2) - \alpha\alphabar \right)
					& \textup{if }\ell\textup{ is ramified in }\OO_{d_1}.
				\end{cases}
		\]
			{Since $\alpha = \left(d_1d_2 - 2t + d2\sqrt{d_1}\right)/(2\sqrt{d_1})$ and $m = \frac14\left(d_1d_2 - (d_1d_2 - 2t)^2\right)$,} in both cases the formula simplifies to $\beta\betabar = \Norm(\frakr\frakD^{-1})m\ell^{-r}.$  From these conditions, it is clear that $\frakb$ is integral, invertible, has norm $m\ell^{-r}$, and is equal to $\frakr$ in $\Pic(\OO_{d_1})/2\Pic(\OO_{d_1})$.
		\end{proof}
		
	We obtain further conditions on the ideal $\frakb$ by using the condition that the index of $\Z\oplus\Z[\alpha,\beta]$ in $(\Q\oplus\Q[\alpha,\beta])\cap R(\fraka)$ is not divisible by any prime $p\neq\ell$.
	
		\begin{lemma}\label{lem:index-unram}
			Let $p$ be a prime that divides $\textup{gcd}(m, f_2)$ and does not divide $d_1$. Then $p$ divides the index of $\Z\oplus\Z\phi_{\fraka\frakc}$ in $\left(\Q\oplus\Q\phi_{\fraka\frakc}\right)\cap R(\fraka\frakc)$ if and only if $\frakb\subset p\OO_{d_1}$.
		\end{lemma}

		\begin{proof}
			It is straightforward to show that a prime $p'$ divides the index if and only if
			\begin{align*}
				\frac{d_2 - p'd_2}{2p'^2} + \frac1{p'}[\alpha,\beta] &=
				\left[\frac{d_2 - p'd_2}{2p'^2} + \frac{d_1d_2 - 2t + d_2\sqrt{d_1}}{2p'\sqrt{d_1}}, \frac{\beta}{p'} \right]\\
				& = \left[\frac{p'(d_1d_2 - 2t) + d_2\sqrt{d_1}}{2p'^2\sqrt{d_1}}, \frac{\beta}{p'}\right] \in R(\fraka)\otimes\Z_{p'}.
			\end{align*}
			Consider a prime $p$ satisfying the assumptions of the lemma.  Since $p$ divides $m$ and $f_2$,
			\[
				\alpha' := \frac{p(d_1d_2 - 2t) + d_2\sqrt{d_1}}{2p\sqrt{d_1}} =
				\frac{1}{\sqrt{d_1}}\left(
				\frac{d_1d_2(p - 1) - 2tp}{2p} +
				\frac{d_2}{p}\frac{d_1 + \sqrt{d_1}}{2}\right) \in p\frakD^{-1}.
			\]
			If $\frakb = \beta\frakr\frakD\frakabar^{-1}\fraka \subseteq p\OO_{d_1}$, then we have that $\frac{\beta}{p}\in\frakr^{-1}\frakD^{-1}\frakabar\fraka^{-1}$, and vice versa.  As $p\nmid d_1$, this happens exactly when $\left[\frac{\alpha'}{p}, \frac{\beta}{p}\right]\in \left(R(\fraka)\otimes\Z_p\right)$, which, as stated above, is equivalent to $p$ dividing the index.
		\end{proof}

		\begin{lemma}\label{lem:index-ram}
			Let $p$ be a prime that divides $\textup{gcd}(m, d_1, f_2)$ that does not divide $\ell f_1$.  If $\frakb\subset \pp^3\OO_{d_1}$, where $\pp$ is the unique prime lying over $p$, then $p$ divides the index of $\Z\oplus\Z\phi_{\fraka\frakc}$ in $\left(\Q\oplus\Q\phi_{\fraka\frakc}\right)\cap R(\fraka\frakc)$.
		\end{lemma}
		\begin{proof} 		
			If $\frakb\subset \pp^3\OO_{d_1}$, then the same arguments as above show that $\alpha' \in p\frakD^{-1}$ and $\beta\in\frakp^3\frakr^{-1}\frakD^{-1}\frakabar\fraka^{-1}$.  To show that $[\alpha',\beta]\in p(R(\fraka)\otimes\Z_p)$, we are left to show that $v_{\pp}(\alpha' - \lambda\beta)\geq 2$.

			We first prove this when $p$ is odd.  Since $p\nmid f_1$, $v_{\pp}(\beta) = v_{\pp}(\lambda\beta)\ge 3 - v_{p}(d) = 2$.   Note that
			\[
				\Norm(\alpha') =
				\frac{-1}{4p^2d_1}\left(p^2(d_1d_2 - 2t)^2 - d_2^2d_1\right) =
				\frac{m}{d_1} + \frac{d_2(d_2 - p^2)}{4p^2}.
			\]
			Given that $p| f_2$ and $p||d_1$, we have $v_p(\Norm(\alpha')) \geq 2$ and so $v_{\pp}(\alpha')\geq 2$.  By the strong triangle equality $v_{\pp}(\alpha' - \lambda\beta)\geq 2$.

			Now consider the case when $p = 2$.  As above, we have that $v_{\pp}(\beta) \ge 3 - v_p(d_1)$, which is non-negative.  We may rewrite
			\begin{equation}\label{eq:Norm-alpha'}
				\Norm(\alpha') = \frac{m}{d_1} + \frac{d_2}{4}\left(\frac{d_2}{4} - 1\right).
			\end{equation}
			Since $2| f_2$, $4$ divides $\frac{d_2}{4}\left(\frac{d_2}{4} - 1\right)$.  Thus, either $v_{\pp}(\alpha'), v_{\pp}(\lambda\beta) \geq 2$, in which case the proof proceeds as above, or $v_{\pp}(\alpha') = v_{\pp}(\lambda\beta)$.  Assume we are in the latter case.  Since $\#\OO/\pp = 2$, we have $v_\pp(\alpha' - \lambda\beta) > v_{\pp}(\lambda\beta) = v_{\pp}(\beta)$.  This completes the proof unless $v_{\pp}(\beta) = v_{\pp}(\alpha') = 0$.  If $v_{\pp}(\beta) = 0$, then $v_p(d_1) = 3$ and $v_p(m) = 3$.  By assumption, $(d_1 , -m)_2 = (d_2, -m)_2 = 1$ so we must have that $\frac{d_2}{4} \equiv 1 \pmod 8.$  Thus, $\Norm(\alpha') \equiv \Norm(\beta) \pmod 8$.  A calculation shows that this gives $v_\pp(\alpha' - \lambda\beta) \geq 2.$
	\end{proof}		

	Together, these two lemmas prove:
	\begin{prop}
		If there exists a prime $p|\textup{gcd}(m, f_2)$, $p\neq \ell$ such that either
		\begin{enumerate}
			\item $p$ is inert in $\OO_{d_1}$, or
			\item $\textup{gcd}(p,f_1)=1$ and $p$ is ramified in $\OO_{d_1}$ and $v_p(m) > 2$,
		\end{enumerate}
		then $S^t_{n,m}(E(\tau_1)) = \varnothing$ for all $\tau_1$.
	\end{prop}

	\begin{proof}
		If $p|\textup{gcd}(m, f_2)$ and $p$ is inert in $\OO_{d_1}$ then any integral ideal $\frakb$ with norm $m\ell^{-r}$ will be contained in $p\OO_{d_1}$.  Thus there are no embeddings of $\OO_{d_1}$ which are optimal at $p$ and thus the set $S^t_{n,m}(E(\tau_1))$ is empty.

		Similarly, if $(p,f_1)=1$ and $p$ is ramified in $\OO_{d_1}$, with $v_p(m) > 2$, then any integral ideal $\frakb$ with norm $m\ell^{-r}$ will be divisible by $\pp^3$ and again there are no embeddings of $\OO_{d_1}$ which are optimal at $p$.
	\end{proof}
	
	We have obtained a map from $S^t_{n,m}(E(\tau_1))$ to invertible integral ideals in $\OO_{d_1}$ satisfying conditions~\eqref{eq:cond-on-b1}, \eqref{eq:cond-on-b2}, and the additional conditions:
	\begin{align}
		 	\frakb &\not\subseteq p\OO_{d_1}, \;
		 		\textup{for any }p|(m,  f_2), p\ne \ell, p\nmid d_1 \label{eq:frakb-divisibility-inert} \\
		 	\frakb &\not\subseteq \frakp^3\OO_{d_1}, \;
		 		\textup{for any }\frakp|p|(d_1, m,  f_2), p\ne \ell, p\nmid  f_1.\label{eq:frakb-divisibility-ram}
	\end{align}
	
	  Note that the codomain does not depend on $\fraka$, and hence is independent of $\tau_1$.  Thus, we can extend the domain of the map to $\bigcup_{\tau_1} S^t_{n,m}(E(\tau_1))$.

      \begin{remark}
          {Note that the codomain of the map is explicitly computable.  This follows from the results in~\S\ref{sec:Background}.}
        \end{remark}
	
	In the rest of the section, we will prove that for each ideal $\frakb$ that satisfies~\eqref{eq:cond-on-b1},\eqref{eq:cond-on-b2},\eqref{eq:frakb-divisibility-inert} and~\eqref{eq:frakb-divisibility-ram}, there is an algorithm that computes whether $\frakb = \beta\frakr^{-1}\frakD\frakabar^{-1}\fraka$ for some $\beta,\fraka$ such that $[\alpha,\beta]\in R(\fraka)\cap S^t_{n,m}(E(\fraka))$.  If so, the algorithm also computes the size of the fiber lying over $\frakb$, i.e. the number of pairs $(\beta,\fraka)$ such that the above condition holds.  This will show that $\sum_{\tau_1} \#S^t_{n,m}(E(\tau_1))$ is a weighted sum on invertible integral ideals, namely each ideal satisfying~\eqref{eq:cond-on-b1},\eqref{eq:cond-on-b2},\eqref{eq:frakb-divisibility-inert} and~\eqref{eq:frakb-divisibility-ram} is weighted by the size of its fiber, and every ideal not satisfying one of~\eqref{eq:cond-on-b1},\eqref{eq:cond-on-b2},\eqref{eq:frakb-divisibility-inert} and~\eqref{eq:frakb-divisibility-ram} is weighted by $0$.  In the case where $m$ is coprime to $f_1$, we will show that this weighted sum agrees with the formula given in the statement of Theorem~\ref{thm:DeterminingSnm}.

	Let $\frakb$ be an invertible integral ideal satisfying~\eqref{eq:cond-on-b1} and~\eqref{eq:cond-on-b2}.  Since $\frakb\sim \frakr\pmod{2\Pic(\OO_{d_1})}$, there exists an integral invertible ideal $\fraka$, well-defined as an element of $\Pic(\OO_{d_1})/(\Pic(\OO_{d_1}))[2]$, and an element $\gamma_{\fraka} \in\Q(\sqrt{d_1})$ such that $\frakb = \gamma_{\fraka} \frakr^{-1}\fraka\frakabar^{-1}$.  (Since $\Pic(\OO_{d_1})$ is finite, there is an algorithm to find $\fraka$, and thus an algorithm to compute $\gamma_{\fraka}$.) For any $[\frakc]\in\Pic(\OO_{d_1})[2]$, let $\epsilon_{\frakcbar^2}/\Norm(\frakcbar)$ be a principal generator for $\frakcbar/\frakc$, and let $\gamma_{\fraka\frakc} := \gamma_{\fraka}\epsilon_{\frakcbar^2}/\Norm(\frakcbar)$.  Then $\frakb$ is also equal to $\gamma_{\fraka\frakc} \frakr^{-1}\fraka\frakc(\frakabar\frakcbar)^{-1}$.  Let $\beta_{\fraka\frakc} = \gamma_{\fraka\frakc}/\sqrt{d_1}$.  Since $\frakb$ is integral, $\beta_{\fraka\frakc}\in\frakr\frakD^{-1}\overline{\fraka\frakc}(\fraka\frakc)^{-1}$.
		
	By the arguments at the beginning of the proof of Proposition~\ref{prop:obtainingb}, we see that if $\frakb$ is the image of an element in $R(\fraka\frakc)$, then the element must be of the form
	\[
		\phi_{\fraka\frakc} :=
		\left[\frac{(d_1d_2 - 2t) + d_2\sqrt{d_1}}{2\sqrt{d_1}},
			\omega\beta_{\fraka\frakc} \right],
	\]
	where $\omega\in\OO_{d_1}^{\times}$.  As $\alpha := \frac{(d_1d_2 - 2t) + d_2\sqrt{d_1}}{2\sqrt{d_1}} \in\frakD^{-1}$ and $\omega\beta_{\fraka\frakc}\in\frakr^{-1}\frakD^{-1}\overline{\fraka\frakc}(\fraka\frakc)^{-1}$, it is left to determine
	\begin{enumerate}
		\item if $\alpha - \omega\lambda_{\fraka\frakc}\beta_{\fraka\frakc} \in\OO_{d_1}$, and
		\item if $\Z\oplus\Z\phi_{\fraka\frakc}$ in $\left(\Q\oplus\Q\phi_{\fraka\frakc}\right)\cap R(\fraka\frakc)$ has index a power of $\ell$.
	\end{enumerate}
	If both conditions hold, then $\frakb$ is the image of $\phi_{\fraka\frakc} \in R(\fraka\frakc)$.

	We first calculate the number of $\phi_{\fraka\frakc}$ satisfying Condition (1).
		
{\bf Condition (1):} Henceforth, we will assume, without loss of generality, that $\fraka$ and $\frakc$ are coprime to $d_1$.  As a result, we have $\lambda_{\fraka\frakc} = (\prod_{\pp|\frakc}M(\pp)^{v_{\pp}(\frakc)}) \lambda_\fraka \equiv \lambda_\fraka \pmod{d_1}$. Thus
	\[
		\omega\lambda_{\fraka\frakc}\gamma_{\fraka\frakc} \equiv
		\omega(\epsilon_{\frakcbar^2}(\prod_{\pp|\frakc}M(\pp)^{v_{\pp}(\frakc)})
		/\Norm(\frakc))
		\lambda_{\fraka}\gamma_{\fraka} \pmod{\frakD}.
	\]
	Since $\alpha - \omega\lambda_{\fraka\frakc}\beta_{\fraka\frakc} \in \OO_{d_1}$ if and only if $\alpha \sqrt{d_1} - \omega\lambda_{\fraka \frakc}\gamma_{\fraka \frakc} \in \frakD$, this in turn occurs if and only if
	\[
		a - \omega(\epsilon_{\frakcbar^2}(\prod_{\pp|\frakc}M(\pp)^{v_{\pp}(\frakc)})/\Norm(\frakc))
		\lambda_{\fraka} \gamma_{\fraka} \in\frakD,
	\]
where $a := \alpha\sqrt{d_1}.$  Note that $\Norm(\lambda_{\fraka}\gamma_{\fraka}) \equiv \Norm(a) \pmod{d_1}$ and that  $\omega\epsilon_{\frakcbar^2}(\prod_{\pp|\frakc}M(\pp)^{v_{\pp}(\frakc)})/\Norm(\frakcbar)$ has norm congruent to $1$ modulo $d_1$.  To complete the proof, we  need the following two lemmas. The first lemma shows that, as $\frakc$ ranges over all elements of $\Pic(\OO)[2]$ and $\omega$ ranges over all elements of $\OO^{\times}$, the element $(\epsilon_{\frakcbar^2}(\prod_{\pp|\frakc}M(\pp)^{v_{\pp}(\frakc)})/\Norm(\frakc))\omega$ ranges over all congruences classes modulo $\frakD$ that have norm congruent to $1\pmod d$.  The second lemma counts the number of such congruence classes of elements $\gamma$ such that $\alpha - \gamma\beta\in \OO_{d_1}$.

	\begin{lemma}\label{lem:norm1elts}
		Let $\OO = \OO_d$.  Then the morphism
		\[
			\Pic(\OO)[2] \longrightarrow \frac{\left\{\gamma \in \left(\OO/\frakD\right)^{\times} : \Norm(\gamma) \equiv 1 \bmod d \right\}}{\OO^{\times}}
		\]
		that sends $[\frakc] \mapsto \epsilon_{\frakc^2}\prod_{\pp|\frakc}M(\pp)^{v_{\pp}(\frakc)}/\Norm(\frakc)$, where $\frakc$ is any integral representative coprime to $\frakD$, is an isomorphism.
	\end{lemma}

	\begin{proof}
        An easy computation shows that the map is a well-defined group homomorphism.  We will show that the groups have the same cardinality, and then show that the map is surjective.  {When $d = -3$ or $-4$, one can check that both sides are trivial.  Henceforth we assume that $d< -4$ so that $\OO^{\times} = \pm1$.}

		  The exact sequence
		\[
			0 \to \Pic(\OO)[2] \to \Pic(\OO) \stackrel{\times2}{\longrightarrow}\Pic(\OO) \to \Pic(\OO)/2\Pic(\OO) \to 0
		\]
		shows that $\#\Pic(\OO)[2] = \#\left(\Pic(\OO)/2\Pic(\OO)\right)$, which by~\cite[Prop. 3.11 \& Thm 7.7]{Cox-PrimesOfTheForm} is equal to $2^{k - 1}$ (see ~\S\ref{sec:Background} for the definition of $k$).  Let $\gamma := \gamma_0 + \gamma_1 \frac{d + \sqrt{d}}{2} \in \OO$ be such that $\Norm(\gamma) \equiv \gamma_0^2 + \frac{d(d-1)}{4}\gamma_1^2 \equiv 1 \pmod d$.  We need to determine the number of $\gamma$ modulo $\frakD$ that satisfy this equation.  For each odd prime $p|d$, we have $2$ choices for $\gamma \bmod \frakD_p$, namely $\gamma\equiv \pm 1\bmod \frakD_p$.  When $p = 2|d$ we have the following possibilities
		\[
			\begin{array}{ll}
				d \equiv\;\; 4 \pmod{16} &
				\gamma\equiv 1\bmod\frakD_2,\\
				d \equiv 12\pmod{16} &
				\gamma\equiv 1,\textup{ or }
				\frac12(d + \sqrt{d})\bmod\frakD_2,\\
				v_2(d) = 3 \textup{ or }4 &
				\gamma \equiv \pm 1 \bmod \frakD_2,\\
				v_2(d) \geq 5&
				\gamma \equiv \pm 1 \bmod\frakD_2 \textup{ or }\gamma \equiv \frac12(d + \sqrt{d}\pm\sqrt{4 + d})\bmod\frakD_2.\\
			\end{array}
		\]
		Note that if $v_2(d) \geq 5$ then $1 + d/4 \equiv 1 \pmod 8$ so $\frac12\sqrt{4 + d}$ exists modulo $2^{v(d) - 1}$.  This case-by-case analysis shows
		\[
		 	\#\left\{\gamma \in \OO/\sqrt{d} :
		\Norm(\gamma) \equiv 1 \pmod d \right\} = 2^{k},
	  	\]
	and so the quotient by $\pm1$ has cardinality $2^{k - 1}$.
			
	\noindent\textit{Proof of surjectivity:}  Fix a nontrivial $\gamma := \gamma_0 + \gamma_1 \frac{d + \sqrt{d}}{2}$ in the codomain, i.e. $\gamma$ is such that $\gamma_0^2 + d\gamma_0\gamma_1 + \gamma_1^2\frac{d^2 - d}{4} \equiv 1\bmod d$ and $\gamma\not\equiv 1\bmod\frakD$.  Then fix non-zero values of $e_0, e_1\in \Z$ and a positive integer $N$ coprime to $d$ such that
		\[
			e_0^2 + de_0e_1 + e_1^2\frac{d^2 - d}{4} = N^2, \quad
			e_0 \equiv N\gamma_0 \pmod{\frac{d}{(2,d)}}, \quad
			e_1\equiv N\gamma_1 \pmod{(2,d)}.
		\]
		(This is possible because conics satisfy weak approximation.)  Define $\epsilon := \frac{1}{(e_0, e_1)}\left(e_0 + e_1 \frac{d + \sqrt{d}}{2}\right)$.  Then $\Norm(\epsilon) = \left(\frac{N}{(e_0,e_1)}\right)^2$, and since $p\nmid\epsilon$, $\epsilon$ must generate a square ideal, say $\frakc^2$.  Additionally, $N = \Norm(\frakc)(e_0, e_1)$.  Therefore, $[\frakc] \mapsto \gamma$, thus completing the proof of surjectivity and of the Lemma.
	\end{proof}

	\begin{lemma}\label{lem:fromidealtoelement}
		Fix $a,b \in \OO_d$.  Assume that $\Norm(a) \equiv \Norm(b) \pmod d$.  If $\Norm(a)$ is coprime to the conductor $f$ of $d$, then there exists an element $c \in \OO$, such that
		\begin{equation}\label{eq:cond-on-c}
			\Norm(c) \equiv 1 \pmod d\textup{ and }
			a - cb\in\frakD.
		\end{equation}
		Moreover, regardless of whether $\Norm(a)$ is coprime to $f$, if there exists $c \in \OO$ satisfying~\eqref{eq:cond-on-c}, then the number of such $c$ modulo $\frakD$ equals $\widetilde{\rho}_d(a_0, a_1)$ where $a := a_0 + a_1\frac{d + \sqrt{d}}{2}$,
		\begin{align*}
			\widetilde{\rho}^{(2)}_d(s,t) & :=
				\left\{
				\begin{array}{ll}
					2 & \textup{if } d \equiv 12 \bmod{16},
						s\equiv t\bmod2\\
					 & \textup{or if } 8\mid d, v(s) \ge v(d) - 2\\
					1& \textup{otherwise}
				\end{array}\right\}
				\cdot\left\{
					\begin{array}{ll}
						2 & \textup{if } 32\mid d, 4\mid (s - 2t)\\
						1& \textup{otherwise}
					\end{array}
					\right\}\\
			\widetilde{\rho}_d(s,t) &:= \widetilde{\rho}^{(2)}_d(s,t)\cdot2^{\#\{p: v_p(s) \ge v_p(d) > 0, p\ne 2\}}.
		\end{align*}
	\end{lemma}
	\begin{remark}\label{rmk:RhoExpression}
		If $\Norm(a)$ is coprime to $f$, then $\widetilde{\rho}_d(a_0,a_1)$ has a simpler expression; it is equal to $2^{\#\{p : p|(d,\Norm(a))\}}$.
	\end{remark}

	\begin{proof}
		First we will prove existence of $c$ in the case where $\Norm(a)$ is coprime to the conductor.  Write $a := a_0 + a_1\frac{d + \sqrt{d}}{2}$ and $b := b_0 + b_1\frac{d + \sqrt{d}}{2}$.  By our assumption on $\Norm(a), \Norm(b)$, we see that $a_0^2 \equiv b_0^2 \pmod{ p^{v(d)}}$ for all odd $p$ and that $a_0^2 - \frac{d}{4}a_1^2 \equiv b_0^2 - \frac{d}{4}b_1^2 \pmod{2^{v(d)}}.$  Since $\Norm(a)$ is coprime to the conductor, we get that $a_0 \equiv \pm b_0 \pmod{p^{v(d)}}$ for all odd primes.  From the description of the elements of norm $1\mod d$ in the proof of Lemma~\ref{lem:norm1elts}, it is clear that there is a $c$, with $\Norm(c) \equiv 1 \pmod d$ such that $a \equiv bc \pmod {\frakD_p}$ for all odd $p$.
			
		It remains to show that if $2\mid d$, then there is such a $c$ that satisfies $a \equiv bc \pmod {\frakD_2}$. The norm congruence as well as the co-primality assumption shows
		\begin{enumerate}
			\item If $d \equiv 12 \bmod{16}$, then either $a_0\equiv a_1 \pmod 2$ and $b_0 \equiv b_1\pmod 2$ (when $\Norm(a) \equiv 0 \pmod 2$ or $a_0\not\equiv a_1 \pmod 2$ and $b_0 \not\equiv b_1\pmod 2$ ($\Norm(a)\equiv 1\pmod 2$).
				
			\item If $v_2(d) = 8$, then $a_0 \equiv \pm b_0 \pmod{4}$ and $a_1\equiv b_1\pmod 2$.
		\end{enumerate}
		The description of possible $c \in \OO$ such that $\Norm(c)\equiv 1 \pmod d$ given in Lemma~\ref{lem:norm1elts} completes the proof of existence.
			
		Henceforth we assume that there exists a $c\in\OO$ such that $\Norm(c) \equiv 1\pmod{d}$ and $a\equiv bc\pmod{\frakD}$.  Since $\Norm(c) \equiv 1 \pmod{d}$ we may replace $b$ with $bc$ and assume that $a\equiv b \pmod{\frakD}$.
			
		Let $c' = c'_0 + c'_1\frac{d + \sqrt{d}}{2}\in \OO$ be such that $a\equiv c'b\pmod{\frakD}$.  Since $a\equiv b \pmod{\frakD}$, we must have $(c' - 1)b\in\frakD$ and $\Norm(c')\equiv 1 \pmod{d}$.  Therefore for all odd primes $p\mid d$, either $c'_0 \equiv 1 \pmod{p^{v(d)}}$ or $v_p(b_0)\ge v_p(d)$, and hence $v_p(a_0)\ge v_p(d)$.  So for odd $p$, there are $2$ choices for $c'$ modulo $\frakD_p$ if $v_p(a_0) \geq v_p(d)$ and otherwise $c'$ is uniquely determined modulo $\frakD_p$.
			
		Now let $p = 2$.  Then we have the following possibilities for $c'$
		\[
			\begin{array}{ll}
				d \equiv\;\; 4 \pmod{16} &
				c'\equiv 1 \bmod{\frakD_2}\\
				d \equiv 12\pmod{16} &
				\textup{either }c'\equiv 1\bmod{\frakD_2},\textup{ or }\\
				& c'\equiv\frac12(d + \sqrt{d})\bmod\frakD_2,
				\textup{ and }a_0 \equiv a_1\pmod 2,\\
				v_2(d) = 3 \textup{ or }4 &
				\textup{either }c'\equiv 1\bmod{\frakD_2},\textup{ or }\\
				& c'\equiv-1\bmod\frakD_2,
				\textup{ and }v_p(b_0),v_p(a_0)\ge v_p(d) - 2,\\
				v_2(d) \geq 5&
				\textup{either }c'\equiv 1\bmod{\frakD_2},\textup{ or }\\
				& c'\equiv-1\bmod\frakD_2,
				\textup{ and }v_p(b_0),v_p(a_0)\ge v_p(d) - 2,\textup{ or }\\
				& c'\equiv \frac12(d + \sqrt{d}\pm\sqrt{4 + d})\bmod\frakD_2
				\textup{ and }a_0 \equiv 2a_1\pmod{4}.
			\end{array}
		\]
		In the last case, when $v_2(d) \geq 5$, if $a_0 \equiv 2a_1\pmod{4}$ \emph{and}  $v_p(a_0) \ge v_p(d) - 2$, then $c'$ may take on all values modulo $\frakD_2$ that have norm congruent to $1$ modulo $d$.  This case-by-case analysis shows that there are $\widetilde{\rho}_d^{(2)}(a_0,a_1)$ choices for $c'$ modulo $\frakD_2$.  This completes the proof.
	\end{proof}
	\begin{remark}
		The assumption that $\Norm(a)$ is coprime to the conductor in Lemma~\ref{lem:fromidealtoelement} is crucial for the existence of $c$.  Let's assume that $p$ divides the conductor and $v_p(d) = 2$.  Then if $p^2 | \Norm(a)$, the congruence $a_0^2 \equiv b_0^2 \pmod{p^2}$, only implies that $a_0 \equiv b_0 \equiv 0 \pmod p$, \emph{not} that $a_0 \equiv \pm b_0 \pmod{p^2}$, which is what is needed to prove existence.
	\end{remark}

	Now we return to the proof of the theorem.  Recall that we are determining for which $[\frakc]\in\Pic(\OO)[2]$ we have
	\begin{enumerate}
		\item $\alpha - (\epsilon_{\frakcbar^2}/\Norm(\frakcbar))
		\omega\lambda_{\fraka}\beta_{\fraka}\in\OO_{d_1}$, and
		\item $\Z\oplus\Z\phi_{\fraka\frakc}$ in $\left(\Q\oplus\Q\phi_{\fraka\frakc}\right)\cap R(\fraka\frakc)$ has index a power of $\ell$.
	\end{enumerate}
	Assume that there exists a $[\frakc]\in\Pic(\OO)$ and $\omega\in\OO_{d_1}^{\times}$  such that $(1)$ holds.  Then Lemmas~\ref{lem:norm1elts} and~\ref{lem:fromidealtoelement} show that there are $\widetilde{\rho}_d(-t,d_2)$ elements of $\Pic(\OO)[2]$ and $\frac{w_1}{2}$ choices for $\omega$ such that $(1)$ is satisfied.  In addition, the proofs of Lemmas~\ref{lem:norm1elts} and~\ref{lem:fromidealtoelement} give an algorithm for testing whether there exists a $[\frakc]\in\Pic(\OO)$ and $\omega\in\OO_{d_1}^{\times}$  such that $(1)$ holds.

	\begin{remark}
		If $m$ and $f_1$ are coprime, then the proofs of these lemmas also show that there does exist a $[\frakc]\in\Pic(\OO)$ and $\omega\in\OO_{d_1}^{\times}$  such that $(1)$ holds.
	\end{remark}
	
	{\bf Condition (2):}  Now we determine which of these $\frac{w_1}{2}\widetilde{\rho}_d(-t,d_2)$ elements also have the desired index.  Condition (2) is satisfied if and only if $\Z[\phi_{\fraka\frakc}]$ is optimally embedded in $R(\fraka\frakc)$ at
every prime $p \ne \ell$.
				
	From the proof of Lemma~\ref{lem:index-unram}, we see that the only primes different from $\ell$ that can divide the index of $\Z\oplus\Z\phi_{\fraka\frakc}$ in $\left(\Q\oplus\Q\phi_{\fraka\frakc}\right)\cap R(\fraka\frakc)$ are the primes $p$ that divide $\textup{gcd}(m,  f_2)$.  If $p$ is such a prime that does not divide $d_1$, then for any preimage of an invertible ideal $\frak{b}$ satisfying conditions~\eqref{eq:cond-on-b1}, \eqref{eq:cond-on-b2}, \eqref{eq:frakb-divisibility-inert}, \eqref{eq:frakb-divisibility-ram}, Lemma~\ref{lem:index-unram} ensures that $p$ does not divide the index.  So if the gcd of $m, d_1,$ and $f_2$ is trivial, then the size of the fiber is equal to the number of elements satisfying condition (1), which was shown above to be either $0$ or $\frac{w_1}{2}\widetilde{\rho}_d(-t,d_2)$ .
		
	Now consider the case that there exists a prime	$p | \textup{gcd}(m, d_1,  f_2)$, $p\ne\ell$. Then we have $v_p(m) \geq 2$.  If $p\nmid  f_1$, then Lemma~\ref{lem:index-ram} shows that $v_p(m) = 2$ and the condition on the index is equivalent to determining whether
		\[
			\alpha' - (\epsilon_{\frakcbar^2}/\Norm(\frakcbar))
			\omega\lambda_{\fraka}\beta_{\fraka} \not\in p\OO_{d_1},
			\textup{ where }
			\alpha' := \frac{p(d_1d_2 - 2t) + d_2\sqrt{d_1}}{2p\sqrt{d_1}}.
		\]
		Note that the norm of $\alpha'\sqrt{d_1}$ (see~\eqref{eq:Norm-alpha'}) and the norm of $\sqrt{d_1}\omega\lambda_{\fraka}\beta_{\fraka}$ are both congruent to $-\ell^r m$ modulo $p^{2 + v_p(d_1)}$.  This is a stronger congruence than the one between $\Norm(\alpha)$ and $\Norm(\beta)$ that we used to invoke Lemma~\ref{lem:fromidealtoelement}.  This stronger congruence allows us to show that exactly half of the $\frac{w_1}{2}\widetilde{\rho}_d(-t,d_2)$ elements of $\Pic(\OO)[2]$ are such that $\alpha' - (\epsilon_{\frakcbar^2}/\Norm(\frakcbar))\omega\lambda_{\fraka}\beta_{\fraka}\not\in p\OO_{d_1}$. (One proves this by using arguments almost identical to those in Lemma~\ref{lem:fromidealtoelement}.)  Furthermore, at each prime, these conditions are independent.  So for each prime dividing $\textup{gcd}(m, d_1,  f_2)$, the size of the fiber consisting of elements satisfying both Conditions (1) and (2), is divided in half.
		
		In summary, the number of $[\frakc]\in\Pic(\OO)[2]$ and $\omega\in\OO_{d_1}^{\times}$ that satisfy conditions $(1)$ and $(2)$ above is bounded above by $\frac{w_1}{2}\widetilde{\rho}_d(-t,d_2)2^{-\#\{p|(m, d_1,  f_2), p\nmid \ell f_1\}}$, with equality if $\textup{gcd}(m,f_1) = 1$.  If $m$ is not coprime to the conductor (in which case $w_1 = 2$), then the number of $[\frakc]$ is either $0$, or is of the form $2^k$ with
		\[
			\#\left\{p : p|\textup{gcd}(m, d_1), p\nmid f_1f_2\right\}
			\leq k
			\leq v_2(\widetilde{\rho}_d(-t,d_2)) - \#\{p|\textup{gcd}(m, d_1,  f_2), p\nmid\ell f_1\}.
		\]
		This proves that $\sum_{\tau_1}\#S^t_{n,m}(E(\tau_1))$ is a weighted sum of ideals with a fixed norm, and the proof gives an algorithm for computing this weighted sum.
		
		From now on assume that $\textup{gcd}(m, f_1) = 1$.  So far we have shown that
		\[
			\sum_{\tau_1}\#S_{n,m}^t(E(\tau_1))	= \frac{w_1\widetilde{\rho}_d(-t,d_2)}{2^{1+\#\{p|\textup{gcd}(m, d_1,  f_2), p\ne\ell\}}}\cdot \#\left\{\frakb\subseteq\OO : \textup{satisfying }~\eqref{eq:cond-on-b1}, \eqref{eq:cond-on-b2}, \eqref{eq:frakb-divisibility-inert}, \eqref{eq:frakb-divisibility-ram}\right\}.
		\]
		
		Note that it follows from Corollary~\ref{cor:genus} that an ideal $\frakb$ of norm $\ell^{-r}m$ satisfies $\frakb \sim\frakr\pmod{2\Pic(\OO_{d_1})}$ if and only if $\widehat{\Psi}_{\ell}(m\ell^{-r}\Norm(\frakr)) = \mathbf{1}$.  In particular, either all ideals of norm $\ell^{-r}m$ satisfy~\eqref{eq:cond-on-b2}, or none do.  Therefore
		\[
			\#\left\{\frakb\subseteq\OO : \textup{satisfying }~\eqref{eq:cond-on-b1}, \eqref{eq:cond-on-b2}, \eqref{eq:frakb-divisibility-inert}, \eqref{eq:frakb-divisibility-ram}\right\} = \begin{cases}
		\frakA(m\ell^{-r}) &\textup{if }\widehat{\Psi}_{\ell}(m\ell^{-r}\Norm(\frakr)) = \mathbf{1},\\
		0 & \textup{otherwise}.
\end{cases}
		\]
		
		By the definition of $q$ in~\S\ref{sec:QuaternionOrders}, $m\ell^{-r}\Norm(\frakr) = -m$ in $\Z/f_1\Z$.  Since $m$ is coprime to $f_1$, $-m$ is congruent to a non-zero square modulo $f_1$ and so $\Psi_p(-m) = 1$ for all $p|f_1$.  (Recall from~\S\ref{sec:Background} that if $p|d_1$ and $p\nmid f_1$, then $\Psi_p(-m) = (d_1, -m)_p$.)  Therefore
		\[
			\#\left\{\frakb\subseteq\OO : \textup{satisfying }~\eqref{eq:cond-on-b1}, \eqref{eq:cond-on-b2}, \eqref{eq:frakb-divisibility-inert}, \eqref{eq:frakb-divisibility-ram}\right\} = \begin{cases}
		\frakA(m\ell^{-r}) &\textup{if }\rho(m)\ne 0,\\
		0 & \textup{otherwise}.
\end{cases}
		\]
		By Remark~\ref{rmk:RhoExpression}, we know that
		\[
		\frac{w_1\widetilde{\rho}_d(-t,d_2)}{2^{1+\#\{p|\textup{gcd}(m, d_1,  f_2), p\ne\ell\}}} = \frac{w_1}{2}2^{\#\{p|\textup{gcd}(m,d_1) : p\nmid f_2\textup{ or }p=\ell\}},
		\]
		which completes the proof.
	\end{proof}

	\begin{prop}\label{prop:LocalFactors}
		Fix a prime $\ell$ and a positive integer $m$ of the form $\frac14(d_1d_2 - x^2)$.  Assume that $m$ and $f_1$ are coprime.  Then
		\[
			\rho(m)\frakA(\ell^{-r}m) = \varepsilon_{\ell}(\ell^{-r}m)
			\prod_{p|m, p\ne\ell}\begin{cases}
				1 + v_p(m) & \left(\frac{d_1}{p}\right) = 1, p\nmid f_2,\\
				2 & \left(\frac{d_1}{p}\right) = 1, p| f_2, \\
				\frac12\left(1 + (-1)^{v_p(m)}\right) &
					\left(\frac{d_1}{p}\right) = -1, p\nmid f_2,\\
				2 & p|d_1, (d_1, -m)_p = 1, p\nmid f_2\\
				1 & p|d_1, (d_1, -m)_p = 1, p| f_2, v_p(m) {\leq} 2\\
				0 & \textup{otherwise},
			\end{cases}
		\]
		where
        \[
            \varepsilon_{\ell}(N) = \begin{cases}
            0 & \textup{if either }N\notin\Z \textup{ or }\ell\nmid d_1 \textup{ and } v_{\ell}(N)\equiv 1 \pmod 2,\\
            2& \textup{if }N\in \Z\textup{ and }\ell|d_1,\\
            1& \textup{otherwise.}
            \end{cases}
        \]
%
	\end{prop}
	\begin{proof}
		Recall that 
		\begin{align*}
			\rho(m) = &
				\begin{cases}
					0 &\textup{if } (d_1, -m)_p = -1 \textup{ for }p|d_1, p\ne \ell d_1,\\
					2^{\#\{p|\textup{gcd}(m, d_1) : p\nmid f_2
					\textup{ or } p= \ell\}} & \textup{otherwise},
				\end{cases}\\
			\textup{ and }\frakA(N) = & \#\left\{
				\begin{array}{ll}
					& \Norm(\frakb) = N, \frakb \textup{ invertible},\\
					\frakb\subseteq\OO_{d_1} : & p\nmid\frakb
					\textup{ for all }p | \textup{gcd}(N, f_2), p\nmid \ell d_1\\
					& \frakp^3\nmid\frakb\textup{ for all }
					\frakp|p|\textup{gcd}(N, f_2, d_1), p\ne\ell
				\end{array}
				\right\}.
		\end{align*}
		If $\ell^{-r}m$ is not an integer, then one can easily see that both sides are $0$.  One can also check that both sides are $0$ if $v_{\ell}(\ell^{-r}m) \equiv 1\pmod2$ and $\ell\nmid d_1$.  So we may assume that $v_{\ell}(m)\geq r$ and that if $v_{\ell}(\ell^{-r}m) \equiv 1 \pmod 2$ then $\ell|d_1$.
		
		Under these assumptions, it is an exercise to show that the number of invertible integral ideals $\frakb$ of norm $\ell^{-r}m$ is equal to
		\[
			\prod_{p|m, p\nmid \ell d_1}
			\begin{cases}
				1 + v_p(m) & \textup{if }\left(\frac{d_1}{p}\right) = 1,\\
				\frac12(1 + (-1)^{v_p(m)}) & \textup{if }\left(\frac{d_1}{p}\right) = -1,
			\end{cases}
		\]
        {(here we use that $\ell$ is not split in $\OO_{d_1}$.)}
		If we impose the condition that $p\nmid\frakb$ for all $p | \textup{gcd}(m\ell^{-r}, f_2), p\nmid \ell d_1$, then for any $p| f_2, p\nmid d_1\ell$ the local factor becomes $0$ if $p$ is inert in $\OO_{d_1}$ and $2$ if $p$ is split.  If we additionally assume that $\frakp^3\nmid\frakb$ for all $\frakp|p|\textup{gcd}(\ell^{-r}m, f_2, d_1), p\ne\ell$, then we introduce a local factor at $p$ that is $0$ when $p|d_1$, $p\neq\ell$ and $v_p(m)\geq3$.  In summary, we have
		\begin{equation}\label{eqn:LocalFactorsForFrakA}
			\frakA(\ell^{-r}m) = \prod_{p\mid m, p \ne \ell}
			\begin{cases}
				1 + v_p(m) &
					\textup{if }\left(\frac{d_1}{p}\right)= 1,
					p\nmid f_2,\\
				2 & \textup{if }\left(\frac{d_1}{p}\right)= 1,
					p\mid f_2,\\
				\frac12\left(1 + (-1)^{v_p(m)}\right) &
					\textup{if }\left(\frac{d_1}{p}\right)= -1,
					p\nmid f_2,\\
				0 & \textup{if }\left(\frac{d_1}{p}\right)= -1,
					p\mid f_2,\\
               { 0} & {\textup{if }p|(d_1,f_2), v_p(m)\geq 3,}\\
                {1} & {\textup{otherwise.}}
			\end{cases}
		\end{equation}
		
		It is clear from the definition of $\rho(m)$ that
		\[
			\rho(m) = \prod_{p|\textup{gcd}(d_1,m)}
				\begin{cases}
					0 & \textup{if }(d_1, -m)_p = -1, p\neq \ell,\\
					1 & \textup{if }(d_1, -m)_p = 1, p|f_2,{p\neq\ell}\\
					2 & {\textup{otherwise}.}
				\end{cases}
		\]
        {Combining these products completes the proof.}
	\end{proof}

\section{Relating $S_{n}^{\Lie}$ to $S_{n}$}\label{sec:LieToEndo}

	We retain the notation from the previous sections.  In this section we prove the following proposition.
	
    \begin{prop}\label{prop:LieToEndo}
        Assume that $d_1$ is fundamental at $\ell$ and that either $\ell>2$ or $\textup{gcd}(d_1,\widetilde{d}_2)$ is odd. Fix $[\tau_1]$ of discriminant $d_1$, and let $m$ be a nonnegative integer divisible by $\ell$.

        \noindent Then if $\ell|f_2$, we have:
        \[
            \#S_{1,m}^{\Lie}(E(\tau_1)/A) =
                \begin{cases}
                    \frac{1}2\#S_{1,m}(E(\tau_1)/\WW) &
                        \textup{if }\ell\nmid \widetilde{d}_2,\\
                    \#S_{1,m}(E(\tau_1)/\WW) & \textup{if }\ell|\widetilde{d}_2,
                \end{cases}
        \]
        and if $\ell\nmid f_2$, we have:
        \[
            \sum_n \#S_{n,m}^{\Lie}(E(\tau_1)/A) =
                \begin{cases}
                    \frac12 \sum_n\#S_{n,m}(E(\tau_1)/\WW) & \textup{if }\ell|d_1\textup{ or }\ell\nmid\widetilde{d}_2,\\
                    \sum_n\#S_{n,m}(E(\tau_1)/\WW) & \textup{if }\ell\nmid d_1, \ell| \widetilde{d}_2.
                \end{cases}
        \]
    \end{prop}

    We first prove a lemma which will be useful in the proof of Proposition~\ref{prop:LieToEndo}.  Recall that if $S_{n,m}(E)\neq\emptyset$ for some elliptic curve $E$, then by~\S\ref{sec:CountingEndomorphisms}, $m = \frac14\left(d_1d_2 - (d_1d_2 - 2t)^2\right)$, where $2t = \Tr((d_1 + \sqrt{d_1})\phi^{\vee})$ for some $\phi \in S_{n,m}(E)$.
    \begin{lemma}\label{lem:mValuation}
        Assume that $s_2 = 0$.  If $S_{1,m}(E(\tau_1))\neq\emptyset$ for some $\tau_1$ of discriminant $d_1$ then $v_{\ell}(m) = 1 $ if $\ell|d_1\tilde d_2$ and $\ell\nmid \gcd(d_1,\tilde d_2)$, and $v_{\ell}(m) > 1$ if $\ell|\gcd(d_2,\tilde d_2)$ and $\ell\neq 2$.
    \end{lemma}
    \begin{proof}
        This follows from a simple calculation.
    \end{proof}
    \begin{proof}[Proof of Proposition~\ref{prop:LieToEndo}]
        Write $E$ for $E(\tau_1)$.  Recall that
        \[
        S_{n,m}^{\Lie}(E/A) = \left\{\phi\in S_{n,m}(E/A) : c(\phi) = \delta, \textup{ and }c(\widetilde\phi) \equiv \widetilde\delta \bmod \mu\right\}.
        \]

        Assume first that $m=0$. Then $d_2 = d_1\ell^{2s_2}$, $\tilde d_2 = d_1$, and we may assume that $\tilde\delta = \frac12(d_1 + \sqrt{d_1})$.  Since $S_{1,m}(E(\tau_1)) = \left\{\frac12(\ell^{2s_2}d_1 \pm\ell^{s_2}\sqrt{d_1})\right\}$, $\#S_{1,m}^{\Lie}(E) = \frac12\#S_{1,m}(E)$ if $\ell\nmid d_1$ and $\#S_{1,m}^{\Lie}(E) = \#S_{1,m}(E)$ if $\ell\mid d_1$.

    Henceforth we assume that $m>0$.  In this case $E$ must be supersingular and we may embed $\End(E)$ into $\Mat_2(\Q(\sqrt{d_1}))$ as described in~\S\ref{sec:QuaternionOrders}.

     Let $\phi\in S_{n,m}(E/\WW)$; then $v_{\ell}(m)\geq 2n - 1$ if $\ell\nmid d_1$ and $v_{\ell}(m) \geq n$ if $\ell | d_1$.  To determine if $\phi\in S_{n,m}^{\Lie}(E)$, we must determine $c(\phi)$ and $c(\widetilde \phi)$.  Since $c$ factors through the formal group, we will consider elements in $\BBl\cap\Z_{\ell}\otimes\End_{A/\mu^r}(E).$  By the theory of complex multiplication, we know that $c([\alpha, 0]) = \alpha$ for any $\alpha \in \OO$.  Since $c$ is a homomorphism, we have
        \[
            (c(\phi) - \delta)(c(\phi^{\vee}) - \delta) =
            \frac14(d_2^2 - d_2) - d_2\delta + \delta^2 = 0,
        \]
        and similarly for $\widetilde{\phi}$ and $\widetilde{\delta}$. In addition, $c(\widetilde \phi) - \delta + c(\widetilde\phi^\vee) - \delta = d_2 - 2\delta = -\sqrt{\widetilde d_2}$.   Therefore, at least one of $c(\widetilde\phi)$, $c(\widetilde\phi_0)$ is congruent to $\widetilde \delta$ modulo $\mu$, and both are congruent to $\widetilde \delta$ modulo $\mu$ if and only if $\ell\mid \widetilde d_2$.  This information is already enough to prove the proposition under the assumption that $\ell\mid f_2$ or $\ell\nmid d_2$.

       Assume that $\ell | d_2$, $s_2 = 0$ and $\ell\nmid d_1$.  Then by Lemma~\ref{lem:mValuation}, $S_{n,m}(E/\WW) = \emptyset$ for all $n > 1$.  Since $A$ is a degree $2$ ramified extension of $\WW$, this implies that $S_{n,m}(E/A) = \emptyset$ for all $n > 2$.  Further, the arguments from the previous paragraph show that $S_{1,m}^{\Lie}(E/A) = S_{1,m}(E/\WW)$.  It remains to consider $S_{2,m}^{\Lie}(E/A)$.  Although $S_{2,m}(E/A) = S_{1,m}(E/\WW)$, which may be nonempty, $\delta$ is not congruent modulo $\mu^2$ to any element of $\WW$.  In particular $\delta$ is not congruent to $c(\phi)$ or $c(\phi^{\vee})$ modulo $\mu^2$.  Therefore $S_{2,m}^{\Lie}(E/A)=\emptyset$, and so $\sum_n\#S_{n,m}^{\Lie}(E/A) = \sum_n\#S_{n,m}(E/\WW)$ as claimed.

        The remaining case is when $s_2 = 0$ (so $\widetilde d_2 = d_2$) and $\ell|\gcd(d_1,\tilde d_2)$.  By assumption, we may restrict to the case where $\ell > 2$, which implies that $A = \WW$.  The above arguments show that $S_{1,m}^{\Lie}(E/A) = S_{1,m}(E/A)$.  In addition, since $v_{\mu}((c( \phi) - \delta) + (c(\phi^\vee) - \delta))) = 1$, at least one of $v_{\mu}(c(\phi) - \delta), v_\mu(c( \phi^\vee) - \delta)$ is $1$; without loss of generality we will assume that $v_\mu(c(\phi^{\vee}) - \delta) = 1$.

        To conclude anything further, we must first determine a more complete description of $c$.  From the results in~\S\ref{sec:CountingEndomorphisms}, we have that any $\phi\in S_{n,m}(E/\WW)$ can be written as
        \[
            \phi = \left[\frac1{2\sqrt{d_1}}\left(d_1d_2 - 2t + d_2\sqrt d_1\right),0\right] + [\beta,0]\cdot[0,1]
        \]
        for some $\beta\in \Q(\sqrt{d_1})$ with $\Norm(\beta) = m/(-qd_1)$. Thus,
        \begin{align*}
            c(\phi) - \delta & = \frac1{2\sqrt{d_1}}\left(d_1d_2 - 2t - \sqrt{d_2d_1}\right) \pm \beta\sqrt{-q} \quad \textup{in } A/\mu^{n}, \textup{ and }\\
            c(\phi^{\vee}) - \delta & = \frac1{2\sqrt{d_1}}\left( 2t -d_1d_2  - \sqrt{d_2d_1}\right) \mp \beta\sqrt{-q} \quad \textup{in } A/\mu^{n}.
        \end{align*}
        Note that $v_{\mu}(\beta) = v_{\mu}(\beta\sqrt{-q}) = v_{\ell}(m) - 1$, and
        \[
            v_{\mu}\left(d_1d_2 - 2t - \sqrt{d_2d_1}\right) +
            v_{\mu}\left(2t - d_1d_2  - \sqrt{d_2d_1}\right) = 2v_{\ell}(m).
        \]
        Since by Lemma~\ref{lem:mValuation}, $v_{\ell}(m) > 1$ and by assumption we have $v_\mu(c(\phi^{\vee}) - \delta) = 1$, we must have
        \[
            v_{\mu}\left(\frac1{2\sqrt{d_1}}\left(d_1d_2 - 2t - \sqrt{d_2d_1}\right)\right) = 2(v_{\ell}(m) - 1).
        \]
        If $v_{\ell}(m) > 2$, then $c(\phi) \equiv \delta\pmod{\mu^n}$ if and only if $n \leq v_{\ell}(m) - 1$.

        Assume that $v_{\ell}(m) = 2$ and recall that $-q\notin\Q_{\ell}^{\times2}$.  If $d_1d_2 \in \Q_{\ell}^{\times2}$, then $c(\phi) - \delta \in \mu^2$ only if $\beta\in \mu^2$.  However, $v_{\mu}(\beta) = v_{\ell}(m) - 1 = 1$, so we must have $c(\phi) - \delta \in \mu\setminus\mu^2$.  If $d_1d_2\notin \Q_{\ell}^{\times2}$, then since $\ell$ is odd we have $-qd_1d_2\in \Q_{\ell}^{\times2}$.  Therefore, we may rewrite $c(\phi) - \delta$ as a linear combination of the $\Q_{\ell}(\sqrt{d_1})$-linearly independent elements $1$ and $\sqrt{-q}$ as follows:
        \[
            \frac1{2\sqrt{d_1}}(d_1d_2 - 2t) + \sqrt{-q}\left(\pm \beta - \sqrt{\frac{d_2}{-4qd_1}}\sqrt{d_1}\right).
        \]
        From this expression, we see that $c(\phi) \equiv \delta\bmod{\mu^2}$ if and only if $v_{\ell}(d_1d_2 - 2t) > 1$ and $\pm\beta/\sqrt{d_1} \equiv \sqrt{\frac{d_2}{-4qd_1}}\bmod \mu$, which would imply that $\Norm(\beta) \equiv d_2/4q\bmod{\ell^2}$.  However, $\Norm(\beta) = m/(-qd_1)$ which, under the assumption that $v_{\ell}(d_1d_2 - 2t) > 1$, implies that $\Norm(\beta) \equiv d_2/(-4q)\bmod{\ell^2}$.  Therefore we must have that $c(\phi) - \delta \in \mu\setminus\mu^2$.

        In summary, we have shown that for $n> 2$, we have $\#S_{n,m}^{\Lie}(E/A) = \frac12\#S_{n+1,m}(E/A)$.  In addition, by Lemma~\ref{lem:mValuation}, $S_{1,m}(E/A) = S_{2,m}(E/A)$.  Hence, we have
        \begin{align*}
            \sum_n\#S_{n,m}^{\Lie}(E/A) & =
            \#S_{1,m}^{\Lie}(E/A) + \sum_{n\geq 2}\#S_{n,m}^{\Lie}(E/A)\\
            &= S_{1,m}(E/A) + \sum_{n\geq 3}\frac12\#S_{n,m}(E/A)\\
            &= \frac12\#S_{1,m}(E/A) + \frac12\#S_{2,m}(E/A)
            + \sum_{n\geq 3}\frac12\#S_{n,m}(E/A),
        \end{align*}
        as desired.
    \end{proof}

\end{document}